\documentclass{amsart}
\newcommand{\RN}[1]{\textup{\uppercase\expandafter{\romannumeral#1}}}
\usepackage{amsmath}
\usepackage{amssymb}
\usepackage{enumerate}
\usepackage[hidelinks]{hyperref}

\usepackage{mathtools}
\usepackage[capitalise]{cleveref}
\numberwithin{equation}{section}

\newcommand\mc{\mathcal}

\newcommand\mb{\mathbb}
\crefname{equation}{}{}
\newtheorem{theorem}{Theorem}[section]
\newtheorem{lemma}[theorem]{Lemma}
\newtheorem{proposition}[theorem]{Proposition}
\newtheorem{corollary}[theorem]{Corollary}

\theoremstyle{definition}

\newtheorem{definition}[theorem]{Definition}
\newtheorem{example}[theorem]{Example}

\theoremstyle{remark}
\newtheorem{remark}[theorem]{Remark}
\usepackage{epigraph}
\usepackage{comment}
\usepackage{tikz}
\usepackage{tikz-cd}

\usepackage{enumitem}
\title{Hilbert schemes on blowing ups and the free Boson}
\begin{document}
\author{Yu Zhao} 
\address{Beijing Institute of Technology, Beijing, China} 
\email{zy199402@live.com} 
\maketitle
\begin{abstract}
  For different cohomology theories (including Hochschild homology, Hodge cohomology, Chow groups, and Grothendieck groups of coherent sheaves), we identify the cohomology of the moduli space of rank $1$ perverse coherent sheaves on the blow-up of a surface by Nakajima-Yoshioka \cite{1050282810787470592} as the tensor product of the cohomology of Hilbert schemes with a Fermionic Fock space. As the stable limit, we identify the cohomology of Hilbert schemes of the blow-up of a surface as the tensor product of the cohomology of Hilbert schemes on the surface with a Bosonic Fock space. The actions of the infinite-dimensional Clifford algebra and Heisenberg algebra are all given by geometric correspondences.
  \end{abstract}
  
\section{Introduction}
\subsection{Motivation and the Main Theorem} Let $S$ be a smooth projective surface over $\mb{C}$. The famous Göttsche's formula computed the Poincaré polynomials of Hilbert schemes of points on $S$ \cite{Göttsche1990}: 
\begin{equation}
    \label{eq:go}
  \sum_{n=0}^{\infty}q^{n}P_{t}(S^{[n]})=\prod_{m=1}^{\infty}\frac{(1+t^{2m-1}q^{m})^{b_{1}(S)}(1+t^{2m+1}q^{m})^{b_{3}(S)}}{(1-t^{2m-2}q^{m})^{b_{0}(S)}(1-t^{2m}q^{m})^{b_{2}(S)}(1-t^{2m+2}q^{m})^{b_{4}(S)}},
\end{equation}
where $b_{i}(S)$ is the $i$-th Betti number. 
\begin{enumerate}
\item By letting $t=1$, the formula above is the character of a super-Fock space. Nakajima \cite{nakajima1997heisenberg} and Grojnowski \cite{grojnowski1996instantons} realized the cohomology group of the Hilbert scheme as the Fock space of a super-Heisenberg algebra through geometric correspondences.

  By replacing the cohomology with the Chow group, the Heisenberg (and moreover, Virasoro and $W_{1+\infty}$) algebra action was constructed by Maulik-Negut \cite{maulik2020}. However, as the Chow group of a surface can be infinite-dimensional (like $K3$ surfaces), the Chow groups of all $S^{n}$ are involved to determine the Chow group of Hilbert schemes, and the representation theory is also more complicated (for example, $sp_{2\infty}$ will naturally appear).  

\item Let $\hat{S}$ be the blow-up of $S$ at a point $o$. By Göttsche's formula, we have
  \begin{equation*}
    \sum_{n=0}^{\infty}q^{n}P_{t}(\hat{S}^{[n]})=\prod_{m=1}^{\infty}\frac{1}{1-t^{2m}q^{m}} \sum_{n=0}^{\infty}q^{n}P_{t}(S^{[n]}),
  \end{equation*}
  where the ratio
  \begin{equation*}
    \prod_{m=1}^{\infty}\frac{1}{1-t^{2m}q^{m}}
  \end{equation*}
  is the character of a Bosonic Fock space. Thus, it suggests we have a canonical isomorphism
  \begin{equation}
    \label{eq:prep}
    \bigoplus_{n=0}^{\infty}H^{*}(\hat{S}^{[n]},\mb{Q})\cong \bigoplus_{n=0}^{\infty}H^{*}(S^{[n]},\mb{Q})\otimes_{\mb{Q}}B,
  \end{equation}
  where $B$ is the Bosonic Fock space $\mb{Q}[y_{1},y_{2},\cdots]$. The formula \eqref{eq:prep} does not depend on the choice of the surface $S$ and the point $o$, and might also hold if we consider the generalized cohomology theory like Hochschild homology, Chow groups and Grothendieck group of coherent sheaves.
\end{enumerate}
 
The purpose of this paper is to realize the expectation \eqref{eq:prep} by proving that:
\begin{theorem}[\cref{thm63}]
  \label{thm:main}
  Let $\mathbb{H}^{*}$ be the functor of Hodge cohomology groups, Hochschild homology groups, Chow groups, or Grothendieck groups of coherent sheaves. We have an identification 
  \begin{equation}
    \bigoplus_{n=0}^{\infty}\mathbb{H}^{*}(\hat{S}^{[n]},\mb{Q})\cong \bigoplus_{n=0}^{\infty}\mathbb{H}^{*}(S^{[n]},\mb{Q})\otimes_{\mb{Q}}B,
  \end{equation}
  where the Heisenberg algebra action on $B$ is induced by geometric correspondences.
\end{theorem} 
\begin{remark}
   More details of our geometric correspondences will be given in \cref{sec:12} to \cref{sec:14}. At the current stage, we do not know the precise relation between our Heisenberg algebra action and that of Nakajima \cite{nakajima1997heisenberg} and Grojnowski \cite{grojnowski1996instantons} in cohomology groups or Maulik-Negu\c{t} \cite{maulik2020} in Chow groups. We expect the study of the affine super-Yangian will shed some light on this problem. See \cref{sec:16} for more discussions.
\end{remark}
\subsection{Stable Limits and Perverse Coherent Sheaves}
\label{sec:12}
\cref{thm:main} follows from regarding both the Bosonic Fock space $B$ and $\mb{H}^{*}(\hat{S}^{[n]},\mb{Q})$ as a stable limit:
\begin{enumerate}
\item We consider the Fermionic Fock space $F$ which is the exterior algebra of
\begin{equation*}
  V=\bigoplus_{i=1}^{\infty} \mb{Q}x_{i}.
\end{equation*}
The degree $1$ operator $K_{1}$
\begin{equation*}
 K_{1}:F_{i}\to F_{i+1},\quad x_{j_{1}}\wedge\cdots\wedge x_{j_{i}}\to  x_{0}\wedge x_{j_{1}+1}\wedge\cdots\wedge x_{j_{i}+1}
\end{equation*}
is injective. Moreover, by direct computation, the stable limit
\begin{equation*}
 F_{\infty}:=\lim_{i\to\infty}F_{i}\cong B.
\end{equation*}
\item Nakajima-Yoshioka \cite{1050282810787470592} considered the moduli space of $l$-stable perverse coherent sheaves $M^{l}(c_{n})$ on $\hat{S}$, where $l\in \mb{Z}$ and 
   \begin{equation*}
      c_{n}=[\hat{S}]+n[o]
    \end{equation*}
    is the Chern character. It has the property that
        \begin{equation*}
      M^{0}(c_{n})\cong S^{[n]}, \quad M^{l}(c_{n})\cong \hat{S}^{[n]} \text{ if } l\gg 0.
    \end{equation*}
    Hence $\hat{S}^{[n]}$ can be regarded as the stable limit of $M^{l}(c_{n})$, $l\to \infty$.
\end{enumerate}
\cref{thm:main} follows from the representation theory on $M^{l}(c_{n})$:
\begin{theorem}[\cref{thm62}]
  \label{thm2}
  We have an identification
  \begin{equation}
    \label{eq:main}
    \bigoplus_{l=0}^{\infty}\bigoplus_{n=0}^{\infty}\mb{H}^{*}(M^{l}(c_{n}),\mb{Q})\cong \bigoplus_{n=0}^{\infty}\mb{H}^{*}(S^{[n]},\mb{Q})\otimes_{\mb{Q}} F
  \end{equation}
  and the action of the infinite-dimensional Clifford algebra on $F$ is given by geometric correspondences.
\end{theorem}

\subsection{A categorical Boson-Fermion correspondence}
The fact that the Bosonic Fock space $B$ is the stable limit of the Fermionic Fock space $F$ is a special case of a categorical Boson-Fermion correspondence which is of independent interest. It reveals an equivalence of certain infinite-dimensional Heisenberg algebra (which we denote as $\mc{H}$) representations and graded infinite-dimensional Clifford algebra (which we denote as $\mc{C}$) representations. The 
bridge between $\mc{C}$ and $\mc{H}$ is an $\mb{Q}$-algebra, which we call $\mc{E}$:
\begin{definition}[\cref{eq:EF}]
  The algebra $\mc{E}$ is defined as a $\mb{Q}$-algebra generated by $K_{1},K_{-1},E,F$ with relations
  \begin{equation*}
    K_{-1}K_{1}=FE=1, \quad K_{-1}E=FK_{1}=0, \quad K_{1}K_{-1}+EF=1.
  \end{equation*}
\end{definition}
  
The relation between $\mc{E}$ with $\mc{C}$ is direct: we consider $\{P_{i},Q_{i}\}_{i\in \mb{Z}_{>0}}$ such that $ P_{1}=EK_{-1}, Q_{1}=K_{1}F$ and
  \begin{align*}
      P_{i}=K_{1}P_{i-1}K_{-1}-EP_{i-1}F, \quad Q_{i}=K_{1}Q_{i-1}K_{-1}-EQ_{i-1}F ,\quad \text{ if } i>1.
  \end{align*}
  Then $P_{i},Q_{i}$ satisfy the relations of infinite-dimensional Clifford algebras (see \cref{prop12} for the details). On the other hand, the Fermionic Fock space $F$ can also be regarded as the highest weight representation of $\mc{E}$ generated by a vector $|1\rangle $ and the relations (see \cref{prop:110} for the details)
  \begin{equation*}
    F|1\rangle =0, \quad K_{-1}|1\rangle =K_{1}|1\rangle =|1\rangle.
  \end{equation*}

  The relation between the representation theory of $\mc{E}$ and $\mc{H}$ is more delicate: roughly speaking, given a graded $\mc{E}$-representation on
  \begin{equation*}
    W=\bigoplus_{i\in \mb{Z}}W_{i}
  \end{equation*}
  such that $K_{1}$ has degree $1$, then $K_{1}$ is injective and induces a stable limit
  \begin{equation*}
    W_{\infty}=\lim_{n\to \infty}W_{n}
  \end{equation*}
  which is endowed with an $\mc{H}$ action if certain admissible property is satisfied. Particularly, we have the following equivalence
  \begin{theorem}[Categorical Boson-Fermion correspondence, \cref{fermion-boson}]
  We have a canonical equivalence between the following categories:
  \begin{enumerate}
  \item the abelian category of admissible $\mc{H}$-modules;
  \item the abelian category of graded admissible $\mc{C}$-modules;
  \item the abelian category of graded admissible $\mc{E}$-modules;
  \item the abelian category of $\mb{Q}$-vector spaces.
  \end{enumerate}
  Here the definitions of admissible modules are defined in \cref{def:ad1,def:ad2,def:ad3}.
  \end{theorem}

\subsection{Derived Grassmannians and incidence varieties} 
  \label{sec:14}
  Due to the categorical Boson-Fermion correspondence, we would like to construct an $\mc{E}$-module structure on
  \begin{equation*}
    \bigoplus_{m=0}^{\infty}\bigoplus_{n=0}^{\infty}\mathbb{H}^{*}(M^{m}(c_{n}), \mathbb{Q}).
  \end{equation*}
  by incidence varieties. It follows from the incidence variety of the derived Grassmannian which is introduced by Jiang \cite{jiang2022grassmanian}: given a locally free sheaf $\mc{I}$ over a variety $X$, the Grassmannian (under the notation of Grothendieck) $\mc{I}$, which we denote as $Gr(\mc{I}, r)$, is the moduli space of quotients
  \begin{equation*}
    \mc{I} \twoheadrightarrow \mc{E}
  \end{equation*}
  where $\mc{E}$ is a rank $r$ locally free sheaf on $X$. Jiang \cite{jiang2022grassmanian} generalized the concept of the Grassmannian and allows $\mc{I}$ to be a perfect complex, i.e., locally a complex of locally free sheaves. A special case we shall care about is when $\mc{I}$ has tor-amplitude $[0,1]$, i.e., it is locally a two-term complex of locally free sheaves. Particularly, in this situation $\mc{I}^{\vee}[1]$ is also a tor-amplitude complex.

 Fixing integers $d_{-},d_{+}\geq 0$, we consider the derived Grassmannians
 \begin{equation*}
   Gr(\mc{I},d_{+}), \quad Gr(\mc{I}^{\vee}[1],d_{-}).
 \end{equation*}
 In \cite{jiang2022grassmanian}, Jiang defined an incidence variety
 \begin{equation*}
   Incid_{X}(\mc{I},d_{+},d_{-})
 \end{equation*}
 which is the fiber product of $Gr(\mc{I},d_{+})$ and $Gr(\mc{I}^{\vee}[1],d_{-})$ over $X$ schematically but with a different perfect obstruction theory structure (i.e. derived structure in derived algebraic geometry). The (virtual) fundamental class of $Incid_{X}(\mc{I},d_{+},d_{-})$ induces correspondences:
 \begin{align*}
   \phi_{d_{+},d_{-}}:\mb{H}^{*}(Gr(\mc{I}^{\vee}[1],d_{-}),\mb{Q})\to \mb{H}^{*}(Gr(\mc{I},d_{+}),\mb{Q}), \\ \psi_{d_{+},d_{-}}: \mb{H}^{*}(Gr(\mc{I},d_{+}),\mb{Q})\to \mb{H}^{*}(Gr(\mc{I}^{\vee}[1],d_{-}),\mb{Q}).
 \end{align*}

 Now we come back to the moduli space of perverse coherent sheaves $M^{m}(c_{n})$. Let $\mc{I}$ be the universal sheaf over $S^{[n]}\times S$ and
 \begin{equation*}
   \mc{I}_{o}:=\mc{I}|_{S^{[n]}\times o}.
 \end{equation*}
 Then by Nakajima-Yoshioka \cite{1050282810787470592}, we have
 \begin{equation*}
  Gr_{S^{[n+\frac{l+l^{2}}{2}]}}(\mc{I}_{o}^{\vee}[1],l)\cong M^{l}(c_{n})\cong Gr_{S^{[n+\frac{-l+l^{2}}{2}]}}(\mc{I}_{o},l).
 \end{equation*}
 Hence the indicence variety of derived Grassmannians induces the maps
 \begin{align*}
  \psi_{l,l-1}:\bigoplus_{n\geq 0}\mb{H}^{*}(M^{l}(c_{n}),\mb{Q})\to \bigoplus_{n\geq 0}\mb{H}^{*}(M^{l-1}(c_{n}),\mb{Q}), \\
  \phi_{l-1,l}:\bigoplus_{n\geq 0}\mb{H}^{*}(M^{l-1}(c_n),\mb{Q})\to \bigoplus_{n\geq 0}\mb{H}^{*}(M^{l}(c_n),\mb{Q}), \\
  \psi_{l,l}:\bigoplus_{n \geq 0}\mb{H}^{*}(M^{l}(c_{n}),\mb{Q})\to \bigoplus_{n\geq 0}\mb{H}^{*}(M^{l}(c_{n-l}),\mb{Q}), \\
  \phi_{l,l}:\bigoplus_{n\geq 0}\mb{H}^{*}(M^l(c_{n-l}),\mb{Q})\to \bigoplus_{n\geq 0}\mb{H}^{*}(M^{l}(c_n),\mb{Q}).
\end{align*}
The $\mc{E}$-representation on the cohomology of moduli space of perverse coherent sheaves follows from the above operators:
\begin{theorem}[\cref{lastmain}]
  \label{thm:15}
  The operators
  \begin{align*}
    E:=\bigoplus_{l>0}\phi_{l-1,l},\quad F:=\bigoplus_{l>0}\psi_{l-1,l}, \quad K_{1}:=\bigoplus_{l\geq 0} \phi_{l,l}, \quad K_{-1}:=\bigoplus_{l\geq 0} \psi_{l,l}
  \end{align*}
  forms an graded admissible $\mc{E}$-representation on
  \begin{equation*}
      \bigoplus_{m=0}^{\infty}\bigoplus_{n=0}^{\infty}\mb{H}^{*}(M^{m}(c_{n}),\mb{Q}).
  \end{equation*}
\end{theorem}

\subsection{Some related work}
\label{sec:16}
\subsubsection{The affine super-Yangian of $gl(1|1)$}
Although not explicitly pointed out in our paper, the representation we constructed is closely related to the quiver Yangian action of Li-Yamazaki \cite{Li2020} and Galakhov-Li-Yamazaki \cite{Galakhov2021}. Here we take an example where $S$ is the equivariant $\mathbb{C}^{2}$ (although it is not projective). Li-Yamazaki \cite{Li2020} defined the affine super-Yangian of $gl(1|1)$, which we denote as
\begin{equation*}
  Y_{\hbar_{1}, \hbar_{2}}(\widehat{gl(1|1)}),
\end{equation*}
as the $\mathbb{Q}(\hbar_{1}, \hbar_{2})$ algebra generated by $e_n^{\circ/\bullet}$, $f_n^{\circ/\bullet}$, $\psi_k^{\circ/\bullet}$, $n \in \mathbb{Z}_{\geq 0}$, $k \in \mathbb{Z}$ with relations:
\begingroup
\renewcommand*{\arraystretch}{1.5}
\begin{equation}\label{Yang_modes}
	\begin{array}{l}
		\left\{e^{\circ}_n,e^{\circ}_k\right\}=\left\{f^{\circ}_n,f^{\circ}_k\right\}=\left[\psi^{\circ}_n,e^{\circ}_k\right]=\left[\psi^{\circ}_n,f^{\circ}_k\right]=\left\{e^{\bullet}_n,e^{\bullet}_k\right\}=\left\{f^{\bullet}_n,f^{\bullet}_k\right\}=\left[\psi^{\bullet}_n,e^{\bullet}_k\right]=\left[\psi^{\bullet}_n,f^{\bullet}_k\right]=0\,,\\
		\left\{e_{n+2}^{\circ},e_{k}^{\bullet}\right\}-2\left\{e_{n+1}^{\circ},e_{k+1}^{\bullet}\right\}+\left\{e_{n}^{\circ},e_{k+2}^{\bullet}\right\}-\frac{\hbar_1^2+\hbar_2^2}{2}\left\{e_{n}^{\circ},e_{k}^{\bullet}\right\}+\frac{\hbar_1^2-\hbar_2^2}{2}\left[e_{n}^{\circ},e_{k}^{\bullet}\right]=0\,,\\
		\left\{f_{n+2}^{\circ},f_{k}^{\bullet}\right\}-2\left\{f_{n+1}^{\circ},f_{k+1}^{\bullet}\right\}+\left\{f_{n}^{\circ},f_{k+2}^{\bullet}\right\}-\frac{\hbar_1^2+\hbar_2^2}{2}\left\{f_{n}^{\circ},f_{k}^{\bullet}\right\}-\frac{\hbar_1^2-\hbar_2^2}{2}\left[f_{n}^{\circ},f_{k}^{\bullet}\right]=0\,,\\
		\left[\psi_{n+2}^{\circ},e_{k}^{\bullet}\right]-2\left[\psi_{n+1}^{\circ},e_{k+1}^{\bullet}\right]+\left[\psi_{n}^{\circ},e_{k+2}^{\bullet}\right]-\frac{\hbar_1^2+\hbar_2^2}{2}\left[\psi_{n}^{\circ},e_{k}^{\bullet}\right]+\frac{\hbar_1^2-\hbar_2^2}{2}\left\{\psi_{n}^{\circ},e_{k}^{\bullet}\right\}=0\,,\\
		\left[\psi_{n+2}^{\circ},f_{k}^{\bullet}\right]-2\left[\psi_{n+1}^{\circ},f_{k+1}^{\bullet}\right]+\left[\psi_{n}^{\circ},f_{k+2}^{\bullet}\right]-\frac{\hbar_1^2+\hbar_2^2}{2}\left[\psi_{n}^{\circ},f_{k}^{\bullet}\right]-\frac{\hbar_1^2-\hbar_2^2}{2}\left\{\psi_{n}^{\circ},f_{k}^{\bullet}\right\}=0\,,\\
		\left[\psi_{n+2}^{\bullet},e^{\circ}_{k}\right]-2\left[\psi_{n+1}^{\bullet},e_{k+1}^{\circ}\right]+\left[\psi_{n}^{\bullet},e_{k+2}^{\circ}\right]-\frac{\hbar_1^2+\hbar_2^2}{2}\left[\psi_{n}^{\bullet},e_{k}^{\circ}\right]-\frac{\hbar_1^2-\hbar_2^2}{2}\left\{\psi_{n}^{\bullet},e_{k}^{\circ}\right\}=0\,,\\
		\left[\psi_{n+2}^{\bullet},f_{k}^{\circ}\right]-2\left[\psi_{n+1}^{\bullet},f_{k+1}^{\circ}\right]+\left[\psi_{n}^{\bullet},f_{k+2}^{\circ}\right]-\frac{\hbar_1^2+\hbar_2^2}{2}\left[\psi_{n}^{\bullet},f_{k}^{\circ}\right]+\frac{\hbar_1^2-\hbar_2^2}{2}\left\{\psi_{n}^{\bullet},f_{k}^{\circ}\right\}=0\,,\\
		\left\{e^{\circ}_n,f^{\circ}_k\right\}=-\psi^{\circ}_{n+k},\quad \left\{e^{\bullet}_n,f^{\bullet}_k\right\}=-\psi^{\bullet}_{n+k},\quad \left\{e^{\circ}_n,f^{\bullet}_k\right\}=\left\{e^{\bullet}_n,f^{\circ}_k\right\}=0\,,\\
		\left[\psi_n^{\circ/\bullet},\psi_k^{\circ/\bullet}\right]=0\,.
	\end{array}
\end{equation}
and higher order relations which we will discuss in our future work.

When $S$ is the equivariant $\mathbb{C}^{2}$, the moduli space of perverse coherent sheaves was described as a quiver-like variety by Nakajima-Yoshioka \cite{1361418519756759552}. The affine super-Yangian $Y_{\hbar_{1}, \hbar_{2}}(\widehat{gl(1|1)})$ acts on the cohomology of moduli spaces of perverse coherent sheaves by Galakhov-Morozov-Tselousov \cite{Galakhov2023}. The operators in \cref{thm:15} actually come from the action of certain elements in $Y_{\hbar_{1}, \hbar_{2}}(\widehat{gl(1|1)})$, which we will discuss in future work.

\subsubsection{Cohomological Hall algebra (COHA)}
An equivalent definition of the affine super-Yangian $Y_{\hbar_{1}, \hbar_{2}}(\widehat{gl(1|1)})$ is the double COHA on the resolved conifold by Soibelmann-Rapcak-Yang-Zhao \cite{rapcak2020cohomological}. Actually, our paper should be regarded as a continuation of Section 7.2.2 of their paper.

By dimension reduction, the affine super-Yangian is also the double COHA of certain one-dimensional objects over a surface, which was defined by Diaconescu-Porta-Sala \cite{diaconescu2022cohomological}. A systematic computation of the double COHA will appear very soon \cite{Hecke}.

\subsubsection{Koseki's categorification}
Another work from which we draw great inspiration is the categorification of the derived category of coherent sheaves on the Hilbert scheme of points by Koseki \cite{koseki2021categorical}. To obtain more precise representation-theoretic relations, we apply the semi-orthogonal decomposition theorem of Jiang \cite{jiang2023derived} with explicit Fourier-Mukai kernels.

\subsection{Notations in this paper}
\subsubsection{Derived algebraic geometry}
All the varieties and derived schemes in this paper will be over $\mathbb{C}$. For a derived scheme $X$, we will denote $Perf(X)$ as the $\infty$-category of perfect complexes on $X$ and $QCoh(X)$ as the $\infty$-category of quasi-coherent complexes.

\subsubsection{Representation theory}
We will use the brackets $[-,-]$ and $\{-,-\}$ to denote the Lie bracket and Poisson bracket, respectively.

We denote $\mc{C}$ as the infinite-dimensional Clifford algebra, i.e., the $\mathbb{Q}$-algebra generated by $\{P_{i}, Q_{i}\}_{i \in \mathbb{Z}_{>0}}$ with relations
\begin{equation}
  \label{eq:cli}
  \{P_{i}, P_{j}\} = 0, \quad \{Q_{i}, Q_{j}\} = 0, \quad \{P_{i}, Q_{j}\} = \delta_{ij}.
\end{equation}
We also denote $\mc{H}$ as the infinite-dimensional Heisenberg algebra, i.e., the $\mathbb{Q}$-algebra generated by $\{H_{i}\}_{i \in \mathbb{Z} - \{0\}}$ with relations
\begin{equation*}
  [H_{i}, H_{j}] = i \delta_{i, -j}.
\end{equation*}
 
We denote $B$ as the Bosonic Fermion space, i.e. the polynomial ring of infinitely many variables $\mb{Q}[y_{1},y_{2},\cdots]$. It is generated by variables $y_i, i\in \mb{Z}_{>0}$ with the relation
\begin{equation*}
  [y_i,y_j]=0, \forall i,j>0.
\end{equation*}
Let $U:=\oplus_{i\in \mb{Z}_{>0}}\mb{Q}y_{i}$, then $B$ is the symmetric algebra of $U$. 
We also denote $F$ as the Fermionic Fock space, i.e. the super-polynomial ring of infinitely many variables $\mb{Q}[x_{1},x_{2},\cdots]$. It is generated by variables $x_i, i\in \mb{Z}_{>0}$ with the relation
\begin{equation*}
  \{x_i,x_j\}=0, \forall i,j>0.
\end{equation*}
Let $V:=\oplus_{i\in \mb{Z}_{>0}}\mb{Q}x_{i}$, then $F$ is the exterior algebra of $V$.
\subsection{Acknowledgements}
I want to thank Hiraku Nakajima, Andrei Negu\c{t}, Francesco Sala, Qingyuan Jiang, Jiakang Bao, Dmitry Galakhov, and Henry Liu for many interesting discussions on the subject.

\section{A categorical Boson-Fermion correspondence}
In this section, we demonstrate an equivalence between specific representations of $\mc{C}$ and $\mc{H}$ which can be regarded as a categorical version of a Boson-Fermion correspondence. The equivalence is established through a new algebra, denoted as $\mc{E}$, generated by $K_{1},K_{-1},E,F$ with the relations 
 \begin{align}
   \label{eq:EF}
    K_{-1}K_{1}=FE=1,\quad
     K_{-1}E=FK_{1}=0, \quad 
     K_{1}K_{-1}+EF=1.
  \end{align}
It naturally emerges from the orthogonal decomposition of a vector space
\begin{equation*}
  V\cong V_{0}\oplus V_{1}.
\end{equation*}
Then the respective operators of projections and embedding
\begin{equation*}
  F:V\to V_{1}, \quad E:V_{1}\to V, \quad K_{-1}:V\to V_{0}, \quad K_{1}:V_{0}\to V
\end{equation*}
satisfy the relations \cref{eq:EF}.
\subsection{Basic properties of \texorpdfstring{$\mc{E}$}{E} and \texorpdfstring{$\mc{E}$}{E}-representations}  Let
  \begin{equation*}
    E_{i}= K_{1}^{i}E, \quad F_{i}=FK_{-1}^{i}, \quad i\geq 0.
  \end{equation*}
Through straightforward calculations, we find that
\begin{lemma}
  \label{lem22}
  The operators $E_{i},F_{i}$ satisfy
  \begin{equation}
    \label{eq:ind}
    \sum_{i=0}^{l}E_{i}F_{i}=1-K_{1}^{l+1}K_{-1}^{l+1},\quad E_{i}F_{i}=K_{1}^{i}K_{-1}^{i}-K_{1}^{i+1}K_{-1}^{i+1};
  \end{equation}
  and
  \begin{align*}
    K_{1}E_{i}=E_{i+1}, \quad K_{-1}E_{i+1}=E_{i}, \quad
   & F_{i}K_{-1}=F_{i+1}, \quad F_{i+1}K_{1}=F_{i}, \\
   &    F_{i}E_{j}=\delta_{i,j}, \quad i,j\geq 0
  \end{align*}
\end{lemma}

Let $P_{1}=EK_{-1}, Q_{1}=K_{1}F$ and
  \begin{align*}
      P_{i}=K_{1}P_{i-1}K_{-1}-EP_{i-1}F, \quad Q_{i}=K_{1}Q_{i-1}K_{-1}-EQ_{i-1}F ,\quad \text{ if } i>1.
  \end{align*}
in \cref{A1} we prove that
\begin{proposition}
  \label{prop12}
  The elements $\{P_{i},Q_{i}\}_{i\in \mb{Z}_{\geq 0}}$ satisfy the relations of $\mc{C}$, i.e. \cref{eq:cli}.
\end{proposition}

\begin{lemma}
  \label{lem17}
  Given an $\mc{E}$-module $W$, the actions of $K_{1}$ and $E$ on $W$ are injective, and the actions of $K_{-1}$ and $F$ on $W$ are surjective. Moreover, as vector spaces,
  \begin{equation*}
    K_{1}W = \ker(F), \quad EW = \ker(K_{-1}), \quad W = K_{1}W \oplus EW.
  \end{equation*}
\end{lemma}
\begin{proof}
  The respective injective and surjective properties follow from
  \begin{equation}
    \label{eq:ker}
    K_{-1}K_{1} = FE = 1.
  \end{equation}
  As $K_{-1}E = FK_{1} = 0$, we have
  \begin{equation*}
    K_{1}W \subset \ker(F), \quad EW \subset \ker(K_{-1}).
  \end{equation*}
  As $K_{1}K_{-1} + EF = 1$, we have
  \begin{equation*}
    \ker(F) = (K_{1}K_{-1} + EF)\ker(F) = K_{1}K_{-1}\ker(F) \subset K_{1}W.
  \end{equation*}
  Hence $\ker(F) = K_{1}W$ and similarly $EW = \ker(K_{-1})$.
  Still by the fact that $K_{1}K_{-1} + EF = 1$, we have
  \begin{align*}
    W = (K_{1}K_{-1} + EF)W = K_{1}K_{-1}W + EFW \subset K_{1}W + EW, \\
    \ker(F) \cap \ker(K_{-1}) \subset \ker(K_{1}K_{-1} + EF) = 0
  \end{align*}
  and thus $W = K_{1}W \oplus EW$.
\end{proof}

\begin{remark}
  Any non-zero $\mc{E}$-module is infinite-dimensional: as $K_{1}$ and $W$ are injective, we have
  \begin{equation*}
    W\cong K_{1}W\oplus EW\cong W\oplus W.
  \end{equation*}
\end{remark}

\subsection{Highest representation theory of \texorpdfstring{$\mc{E}$}{E} and Fermionic Fock space}

\begin{definition}
  Given $\lambda \in \mb{Q}^{*}$, an $\mc{E}$-module $W$ is defined to be the highest weight representation of weight $\lambda$ if there exists $w\in W$ which generates $W$ and 
\begin{equation*}
    Fw=0, \quad K_1w=\lambda w.
  \end{equation*}
The Verma module of the highest weight representation of weight $\lambda$ is denoted as $F_{\lambda,\mc{E}}$, and the highest weight vector is denoted as $0\rangle $. We denote $F_{\mc{E}}:=F_{1,\mc{E}}$.
\end{definition}

Given $\lambda\in \mb{Q}^{*}$, there is an automorphism of $\mc{E}$ which maps
\begin{equation*}
  E\to \lambda E, \quad K_{1}\to \lambda K_{1}, \quad F\to \lambda^{-1} F, \quad K_{-1}\to \lambda^{-1} K_{-1}.
\end{equation*}
Under this auto equivalence, the representation $F_{\lambda,\mc{E}}$ is mapping to $F_{\mc{E}}$. 
  
The Fermionic Fock space $F$ has the basis
  \begin{equation*}
    x_{i_{1}}\wedge\cdots\wedge x_{i_{k}}
  \end{equation*}
  for any sequence $0<i_{1}<i_{2}<\cdots <i_{k}, k\geq 0$. 
Through straightforward calculations, we find that
\begin{lemma}
  \label{lem:rho}
  There is a $\mc{E}$-representation on the Fermionic Fock space $F$ such that
  \begin{align*}
&    K_{1}( x_{j_{1}}\wedge\cdots\wedge x_{j_{k}})=x_{j_{1}+1}\wedge\cdots\wedge x_{j_{k}+1}, \\
&    K_{-1}( x_{j_{1}}\wedge\cdots\wedge x_{j_{k}})=x_{j_{1}-1}\wedge\cdots\wedge x_{j_{k}-1}, \\
&    E(x_{j_{1}}\wedge\cdots\wedge x_{j_{k}})=x_{1}\wedge x_{j_{1}+1}\wedge\cdots\wedge x_{j_{k}+1}, \\
&   F(x_{j_{1}}\wedge\cdots\wedge x_{j_{k}})=\frac{\partial}{\partial x_{1}}(x_{j_{1}-1}\wedge\cdots\wedge x_{j_{k}-1}).
  \end{align*}
  where we set $x_{0}:=0$. For the induced $\mc{C}$-representation on $F$, the action of $P_{i}$ is  wedging $x_{i}$ and the action of $Q_{i}$ is $\frac{\partial}{\partial x_{i}}$.
\end{lemma}

\begin{proposition}
  \label{prop:110}
  We have an isomorphism of $\mc{E}$-representation
  \begin{equation*}
    F_{\mc{E}}\cong F
  \end{equation*}
  \end{proposition}
\begin{proof}
  Let $|1\rangle$ be the highest weight vector of $F_{\mc{E}}$. Then the map
  \begin{equation*}
    |1\rangle \to 1
  \end{equation*}
  induces a morphism from $F_{\mc{E}}$ to $F$. We notice that all the vectors
  \begin{equation*}
    E_{i_{1}} E_{i_{2}-i_{1}}\cdots E_{i_{n}-i_{n-1}}|1\rangle 
  \end{equation*}
  where $0\leq i_{1}\leq i_{2}\leq \cdots \leq i_{n}$
  forms a basis of $F_{\mc{E}}$, and its image is
  \begin{equation*}
    x_{i_{1}}\wedge x_{i_{2}+1}\wedge \cdots \wedge x_{i_{n}+n}.
  \end{equation*}
  Hence the map from $F_{\mc{E}}$ to $F$ is an isomorphism.

   Now we prove that the Fermionic Fock space $F$ is irreducible as a $\mc{E}$-module. By associating all $x_i$ with degree $1$, we give a $\mb{Z}_{\geq 0}$-grading on $F$. Let $a=\oplus_{i=1}^{N} a_{i}\in F$ such that $a_{i}$ has degree $i$, we take a monomial
  \begin{equation*}
    x_{j_{1}}\wedge x_{j_{2}}\wedge \cdots x_{j_{N}}
  \end{equation*}
  in $a_{N}$ and assume its coefficient is $1$. Then
  \begin{equation*}
  F_{j_{N}-j_{N-1}-1}\cdots F_{j_{1}-1}a=1.
\end{equation*}
Thus the module $\mc{E}$ is irreducible.
\end{proof}

\subsection{Admissible representations and Boson-Fermion correspondence}
We consider the category of graded $\mc{E}$-modules (resp. $\mc{C}$-modules), 
\begin{equation*}
  W\cong \bigoplus_{i\in \mb{Z}}W_{i}
\end{equation*}
such that the degree of $E,F,K_{\pm 1}$ are $1,-1,0$ respectively (resp. the degree of $P_{i}$ is $1$ and $Q_{i}$ is $-1$).
\begin{definition}
  \label{def:ad1}
  A $\mb{Z}_{\geq 0}$-graded $\mc{E}$-module
  \begin{equation*}
    W\cong \bigoplus_{i=0}^{\infty}W_{i}
  \end{equation*}
  is called admissible if the action of $K_{-1}$ on $W_{i}$ is identity when $i=0$ and locally nilpotent when $i>0$.
\end{definition}
\begin{remark}
  By definition we also have $K_{1}|_{W_{0}}=id$.
\end{remark}
\begin{definition}
  \label{def:ad2}
  A graded $\mc{C}$-module 
  \begin{equation*}
    W:=\bigoplus_{i\geq 0}W_{i}
  \end{equation*}
  is defined to be admissible if
  \begin{enumerate}
    \item for any $x\in W$, $Q_{m}x=0$ when $m\gg 0$.
  \item We have
    \begin{equation*}
      W_{0}\cong \bigcap_{i>0}ker(Q_{i}).
    \end{equation*}
  \end{enumerate}
\end{definition}

\begin{example}
On the Fermionic Fock space, we associate $x_{i}$ with degree $1$. Then $F$ is both a graded admissible $\mc{E}$-module and a graded admissible $\mc{C}$-module.
\end{example}

\begin{example}
  We consider the Bosonic Fock space
\begin{equation*}
  B:=\bigoplus_{i=0}^{\infty}V\cong \mb{Q}[y_{1},y_{2},\cdots ].
\end{equation*}
It has the $\mc{H}$-representation structure by
\begin{equation*}
  \begin{cases}
    H_{k}\to y_{k}, \quad k>0, \\
    H_{k}\to \frac{\partial}{\partial y_{-k}}, \quad k<0.
  \end{cases}
\end{equation*}
For a sequence of non-negative integers $\lambda=(n_{1},n_{2},\cdots )$ such that $n_{m}=0$ when $m\gg 0$, we denote $y_{\lambda}:=\prod_{i=1}^{\infty}y_{i}^{n_{i}}$, which is a monomial of finite variables.
\end{example}

We also introduce the admissible $\mc{H}$-representations:
\begin{definition}
  \label{def:ad3}
  Given a sequence of non-negative integers $\lambda=(n_{1},n_{2},\cdots )$ such that $n_{m}=0$ when $m\gg 0$, we denote
  \begin{equation*}
   |\lambda|=\sum_{i=1}^{\infty}in_{i}, \quad H_{\lambda}=\prod_{i=1}^{\infty}H_{i}^{n_{i}}, \quad H_{-\lambda}=\prod_{i=1}^{\infty}H_{-i}^{n_{i}}.
  \end{equation*}
  An $\mc{H}$-representation $W$ is called admissible, if for any $x\in W$ there exists an integer $N>0$ such that $H_{-\lambda}x=0$ for any $|\lambda|>N$.
\end{definition}
\begin{remark}
  The $\mc{H}$-representation $B$ is admissible.
\end{remark}
In \cref{sec:grade,sec12,sec:1.3}, we establish the following equivalence
\begin{proposition}
  \label{cor:229}
  There is an equivalence among the abelian category of admissible $\mc{H}$-modules, graded admissible $\mc{C}$-modules, and graded admissible $\mc{E}$-modules. 
\end{proposition}
\subsection{Graded $\mc{E}$-modules}
\label{sec:grade}
We first establish the equivalence between the category admissible graded $\mc{E}$-modules and graded $\mc{C}$-modules. For any $\mc{E}$-module $W$, the morphsim from $\mc{C}$ to $\mc{E}$ induces a $\mc{C}$-representation on $W$.
\begin{proposition}
  \label{prop:213}
  If $W$ is an admissible graded $\mc{E}$-module, the induced $\mc{C}$-module on $W$ is also admissible.
\end{proposition}
\begin{proof}
  First we prove that for any $x$ in $W$, $Q_{m}x=0$ when $m\gg 0$ by induction. If $x\in W_{0}$, it follows from the fact that $Q_{1}x=0$ and $Fx=0$. Assuming it is true for $x\in W_{i}$, then for $x\in W_{i+1}$, we have that for some $N>0$
  \begin{equation*}
    x=\sum_{i=1}^N E_{i}F_{i}x
  \end{equation*}
By \cref{eqqq} and induction, we have 
\begin{equation*}
  Q_{m}E_{i}F_{i}x=E_{i}Q_{m-i}F_{i}x=0
\end{equation*}
when $m\gg 0$. Hence $Q_{m}x=0$ for $m\gg 0$. 
 
Now we prove that
  \begin{equation*}
    W_{0}\cong \bigcap_{i>0} ker(Q_{i}).
  \end{equation*}
  First we notice that for any non-zero $x\in \oplus_{i>0}W_{i}$, there exists $M\geq 0$ such $F_{M}x\neq 0$. Otherwise, by \cref{eq:ind}, for any $M>0$
  \begin{equation*}
    K_{-1}^{M}x=x
  \end{equation*}
  which contradicts the assumption of admissibility. Let $M_{0}$ be the smallest integer such that $F_{M_{0}}x\neq 0$. Then by induction, we have
  \begin{align*}
    Q_{M_{0}+1}x=K_{1}Q_{M_{0}}K_{-1}x=\cdots =K_{1}^{M_{0}}Q_{1}K_{-1}^{M_{0}}x=K_{1}^{M_{0}+1}F_{M_{0}}x.
  \end{align*}
  which is non-zero as $K_{1}$ is injective.
\end{proof}

On the other hand, it is more complicated to construct a $\mc{E}$-representation from a graded $\mc{C}$-module, which we will leave to \cref{prop125}.

\subsection{Filtration and Fermionization}
\label{sec12}
In this subsection we give the functor from admissible $\mc{H}$-representations to graded admissible $\mc{E}$-representations: let $W_{\infty}$ be an admissible $\mc{H}$-representation, we consider
\begin{equation*}
  W_{i}:=\bigcap_{l> i}ker(H_{l}).
\end{equation*}
It satisfies that $W_{i}\subset W_{j}$ for $i\leq j$ and
\begin{equation*}
  W_{\infty}=\bigcup_{i\geq 0}W_{i}.
\end{equation*}
Hence it forms a filtration of $W$. For any $k>0$, we consider operators
\begin{equation*}
  \mc{L}_{k}:=\sum_{i=1}^{\infty}\frac{-1}{k^{i}i!}H_{k}^{i-1}H_{-k}^{i}
\end{equation*}
The admissible property of $W_{\infty}$ makes sure that the action of $\mc{L}_{k}$ on $W_{\infty}$ is well defined. By the commutative property of the Heisenberg algebra, we have
\begin{equation}
  \label{eq:comm}
  H_{k}W_{k}\subset W_{k}, \quad \mc{L}_{k}W_{k}\subset W_{k}.
\end{equation}
In \cref{sec:A2}, we will prove that
\begin{proposition}
  \label{224}
  We have
  \begin{equation}
    \label{eq25}
    (Id-H_{k}\mc{L}_{k})W_{k}\subset W_{k-1}.
  \end{equation}
  Let $E:W_{k}\to W_{k+1}$ be the embedding operator, and $F:W_{k+1}\to W_{k}$ be the operator induced by $Id-H_{k}\mc{L}_{k}$.
  Moreover, we consider the operators
  \begin{equation*}
  K_{1}:=\bigoplus_{k\geq 0}H_{k}, \quad K_{-1}:=\bigoplus_{k\geq 0}\mc{L}_{k},
\end{equation*}
where $H_{0}=\mc{L}_{0}:=id$.  Then $W$ is an graded admissible $\mc{E}$-module under the action of $E,F,K_{1},K_{-1}$. 
\end{proposition}

\subsection{Stable limit and Bosonization}
\label{sec:1.3}
The functor from graded admissible $\mc{E}$-modules to admissible $\mc{H}$-modules is constructed by taking the stable limit:
\begin{definition}
  Given a graded admissible $\mc{E}$-module
  \begin{equation*}
    W:=\bigoplus_{i\geq 0}W_{i}
  \end{equation*}
  we consider the stable limit, 
  \begin{equation*}
    W_{\infty}:=\lim_{n\to \infty}W_{i}
  \end{equation*}
  where the morphism form $W_{i}$ to $W_{i+1}$ is $K_{1}$.
\end{definition}

By induction, there exists a unique series of operators  $\{H_{i,j}\}|_{-i\leq j\leq i, j\neq 0}$ which acts on $W_{i}$ such that
  \begin{equation*}
    H_{i,i}:=K_{1}|_{W_{i}},\quad H_{-i,i}:=-i \sum_{l=0}^{\infty}(K_{1})^{l}K_{-1}^{l+1}|_{W_{i}}.
  \end{equation*}
  and if $|j|\neq i$
  \begin{equation*}
    H_{i,j}:=\sum_{l=0}^{\infty}E_{l}H_{i-1,j}F_{l}.
  \end{equation*}
  Moreover, as
    \begin{align*}
    H_{i+1,j}E=\sum_{l=0}^{\infty}E_{l}H_{i,j}F_{l}E=\sum_{l=0}^{\infty}E_{l}H_{i,j}\delta_{0,l}=E_{l}H_{i,j},
    \end{align*}
    for any $l\neq 0$, the operators $\{H_{i,l}\}_{i\geq 0}$ converges to an operator $H_{l}$ on $W_{\infty}$. In the \cref{sec:A3}, we will prove that
    \begin{proposition}
        \label{prop:228}

      The operators $H_{l}|_{l\in \mb{Z}-\{0\}}$ forms an admissible $\mc{H}$-representation on $W_{\infty}$.
    \end{proposition}

    \begin{proof}[Proof of \cref{cor:229}]
      It follows from \cref{prop:213,prop125,224,prop:228}.
    \end{proof}

\subsection{The Boson-Fermion correspondence}
The Fermion Fock space $F$ is a graded admissible $\mc{E}$-module and the Bosonic Fock space $B$ is an admissible $\mb{H}$-module. We have that
\begin{proposition}
  \label{prop:220}
  The stable limit of the Fermionic Fock space $F$ is the Bosonic Fock space $B$.
\end{proposition}
\begin{proof}
  We construct the mapping from $\rho:F\to B$
\begin{equation*}
  \rho(x_{i_{1}}\wedge x_{i_{2}}\cdots \wedge x_{i_{n}}):=y_{1}^{i_{n}-i_{n-1}-1}y_{2}^{i_{n-1}-i_{n-2}-1}\cdots y_{n}^{i_{1}-1}.
\end{equation*}
It is easy to see the following properties of $\rho$
\begin{enumerate}
\item $\rho$ is invariant under $K_{1}$, i.e. for any $x\in F$,
  \begin{equation*}
    \rho(K_{1}x)=\rho(x).
  \end{equation*}
\item For any $i\geq 0$, $\rho$ induces an isomorphism of vector spaces between $F_{i}$ and
  \begin{equation*}
    B_{i}\cong \mb{Q}[y_{1},\cdots, y_{i}].
  \end{equation*}
\end{enumerate}
hence by taking the limit $i\to \infty$, we prove \cref{prop:220}.
\end{proof}

\begin{theorem}[Categorical Boson-Fermion Correspondence]
  \label{fermion-boson}
  By tensoring the Fermionic or Bosonic Fock space, it induces an equivalence from the abelian category of vector spaces to the abelian category of graded admissible $\mc{C}$-modules or admissible $\mc{H}$-modules.
\end{theorem}

By \cref{cor:229}, \cref{fermion-boson} follows from the following proposition:
\begin{proposition}
  \label{prop:225}
  For any graded admissible $\mc{C}$-module $W$, 
  \begin{equation*}
    W\cong F\otimes_{\mb{Q}}W_{0}.
  \end{equation*}
\end{proposition}
\begin{proof}

  For any $k>0$, we consider
  \begin{equation*}
    \bar{W}^{k}:=\bigcap_{l>k}ker(Q_{l})
  \end{equation*}
  which is a $\mb{Z}_{\geq 0}$-grade vector space and 
\begin{equation*}
  W=\lim_{k\to \infty}\bar{W}^{k}.
\end{equation*}

We notice that 
\begin{equation*}
  F\cong \lim_{k\to \infty}\wedge_{i=1}^{k}(\mb{Q}\oplus \mb{Q}x_{i}).
\end{equation*}
Hence to prove \cref{prop:225}, we only need to prove that for any $k\geq 0$,
\begin{equation*}
  \bar{W}^{k}\cong W_0\otimes_\mb{Q}\bigotimes_{i=1}^{k}\wedge^i(\mb{Q}\oplus \mb{Q}x_i).
\end{equation*}
Moreover, by induction, we only need to prove that 
\begin{equation*}
  \bar{W}^{k+1}\cong \bar{W}^{k}\otimes_{\mb{Q}}(\mb{Q}\oplus \mb{Q}x_{k+1}).
\end{equation*}
as $\bar{W}^0\cong W_0$. It follows from the action of the Clifford algebra generated by $P_{k+1}$ and $Q_{k+1}$.
\end{proof}

\section{Derived categories of derived Grassmannians}
Given a variety $X$, a locally free sheaf $\mc{V}$ and an integer $r\geq 0$, the Grassmannian (on Grothendieck's notation) $Gr_{X}(\mc{V}, r)$ is defined as the moduli space of quotients
\begin{equation*}
  \mc{V}\twoheadrightarrow \mc{E}
\end{equation*}
where $\mc{E}$ is a locally free sheaf of rank $r$.

Jiang  \cite{jiang2022derived,jiang2023derived,jiang2022grassmanian} extended the definition of Grassmannians and flag varieties such that $\mc{V}$ can be any perfect complex. Particularly,  when $\mc{V}$ is a tor-amplitude $[0,1]$ perfect complex, i.e. locally a two-term complexes of locally free sheaves
\begin{equation*}
  \mc{V}\cong \{W\to V\}
\end{equation*}
the derived Grassmannian is quasi-smooth (i.e. the locally complete intersection scheme in the derived algebraic geometry setting). Moreover, Jiang \cite{jiang2023derived} strengthens the result of Toda \cite{MR4608555} for the semi-orthogonal decomposition of derived Grassmannians by given Fourier-Mukai kernels. In this section, we review several main results of Jiang \cite{jiang2022derived,jiang2023derived,jiang2022grassmanian}. 

In this section, we will always assume that $X$ is a derived scheme and $\mc{E}$ is a Tor-amplitude $[0,1]$ perfect complex on $X$ of rank $e$.

\subsection{Derived Grassmannians of a Tor-amplitude $[0,1]$ complex}
\label{sec:2.1}

\begin{definition}[Definition 4.3 of \cite{jiang2022grassmanian}]
  Given an integer $d\geq 0$, we define the derived Grassmannian
  \begin{equation*}
    Gr_{X}(\mc{E},d)
  \end{equation*}
  as the derived stack which parametrized the moduli space of quotients
  \begin{equation*}
    \{\mc{E}\twoheadrightarrow \mc{V}\}
  \end{equation*}
  where $\mc{V}$ is a rank $d$ locally free sheaf. We define
  \begin{equation*}
    pr_{\mc{E},d}:Gr_{X}(\mc{E},d)\to X
  \end{equation*}
  as the projection morphism.
\end{definition}


The concept of the quasi-smooth morphisms (schemes) is a generalization of the locally complete intersection morphisms (schemes) in derived algebraic geometry:

\begin{definition}
  A morphism of derived schemes $f:X\to Y$ is defined to be quasi-smooth if the cotangent complex $L_{f}:=L_{X/Y}$ is of tor-amplitude $[0,1]$. A derived scheme $X$ is defined to be quasi-smooth if the morphism from $X$ to $\mathrm{Spec}(\mb{C})$ is quasi-smooth. For a quasi-smooth morphism $f$ (resp. scheme $X$), we define its virtual dimension, denoted as
  \begin{equation*}
    vdim(f), \quad (\text{resp. } vdim(X))
  \end{equation*}
  as the rank of the cotangent complex $L_{f}$ (resp. $L_{X}$). The relative canonical bundle $K_{f}$ (resp. $K_{X}$) is defined as the determinant of $L_{f}$ (resp. $L_{X}$). (We refer to \cite{SchurgToenVezzosi+2015+1+40} or \cite{jiang2022derived} for the determinant of a tor-amplitude $[0,1]$ complex). 
\end{definition}

\begin{remark}
  Given a quasi-smooth derived scheme $X$ over $\mb{C}$, we say that $X$ is classical if it has a trivial derived structure, i.e. $\pi_{0}(X)\cong X$ where $\pi_{0}(X)$ is the underlying scheme. By Corollary 2.11 of \cite{Arinkin2015} (which we also discussed in the Appendix A of \cite{yuzhaoderived}), $X$ is classical if and only if
\begin{equation*}
  vdim(X)=dim(\pi_{0}(X)).
\end{equation*}
and $X$ is always a locally complete intersection variety if $X$ is classical.
\end{remark}

\begin{theorem}[\cite{jiang2022grassmanian}]
  \label{thm:grass}
  The derived Grassmannian $Gr_{X}(\mc{E},d)$ is a proper derived scheme over $X$ and the projection map
  \begin{equation*}
    pr_{\mc{E},d}:Gr_{X}(\mc{E},d)\to X
  \end{equation*}
  is quasi-smooth with relative virtual dimension is $d(rank(\mc{E})-d)$.
\end{theorem}


For a proper quasi-smooth morphism $f:X\to Y$, the derived pullback and pushforward are all well-defined in the $\infty$-category of perfect complexes
\begin{equation*}
  f^{*}:Perf(Y)\to Perf(X), \quad f_{*}:Perf(X)\to Perf(Y).
\end{equation*}
Let $f_{!}$ be the left-adjoint functor $f^{*}$, and $f^{!}$ by the right-adjoint functor of $f_{*}$.  By the Grothendieck duality, we have
\begin{equation}
  \label{dual}
  f_{!}\cong f_{*}(-\otimes K_{Y/X})[vdim(f)], \quad f^{!}\cong  (-\otimes K_{Y/X})\circ f^{*}[vdim(f)].
\end{equation}

\subsection{The incidence varieties}
\label{sec:2.2}
We notice that if $\mc{E}$ has tor-amplitude $[0,1]$, the shifted dual $\mc{E}^{\vee}[1]$ also has tor amplitude $[0,1]$. 

\begin{definition}[Incidence correspondence, Definition 2.7 of \cite{jiang2023derived}]
  Given $(d_{+},d_{-})\in \mb{Z}_{\geq 0}^{2}$, the incidence scheme $Inc_{X}(\mc{E},d_{+},d_{-})$ can be defined in two ways, which are equivalent due to \cite{jiang2022grassmanian}:
  \begin{enumerate}
  \item Over the projection morphism $pr_{+}:Gr_{X}(\mc{E},d_{+})\to X$, let $\mc{E}_{+}$ be the fiber of the universal quotient
     \begin{equation*}
    \theta_{+}: pr^{*}_{+}\mc{E}\to \mc{V}^{+}.
  \end{equation*}
  we define
  \begin{equation*}
    Inc_{X}(\mc{E},d_{+},d_{-}):=Gr_{Gr_{X}(\mc{E},d+)}(\mc{E}_{+}^{\vee}[1],d_{-}).
  \end{equation*}
\item Over the projection morphism $pr_{-}:Gr_{X}(\mc{E}^{\vee}[1],d_{-})\to X$, let $\mc{E}_{-}$ be the fiber of the universal quotient
   \begin{equation*}
    \theta_{-}: pr^{*}_{-}\mc{E}^{\vee}[1]\to \mc{V}^{-\vee}.
  \end{equation*}
  we define
   \begin{equation*}
  Inc_{X}(\mc{E},d_{+},d_{-}):= Gr_{Gr_{X}(\mc{E}^{\vee}[1],d_{-})}(\mc{E}_{-}^{\vee}[1],d_{+}).
  \end{equation*} 
  \end{enumerate}
 The above definitions induce canonical projection morphisms:
  \begin{equation*}
    r^{+}_{d_{+},d_{-}}:Inc_{X}(\mc{E},d_{+},d_{-})\to Gr_{X}(\mc{E},d_{+}), \quad r^{-}_{d_{+},d_{-}}:Inc_{X}(\mc{E},d_{+},d_{-})\to Gr_{X}(\mc{E}^{\vee}[-1],d_{-}).
  \end{equation*}
  which we will also abberivate as $r^{+},r^{-}$ respectively.
\end{definition}

By \cref{thm:grass}, $r^{+}$ and $r^{-}$ are both quasi-smooth and proper, and the relative virtual dimensions are
\begin{equation*}
  vdim(r^{+})=d_{-}(-e+d_{+}-d_{-}), \quad vdim(r^{-})=d_{+}(e+d_{-}-d_{+}).
\end{equation*}
The virtual dimension of the incidence variety $Inc_{X}(\mc{E},d_{+},d_{-})$ over $X$ is
  \begin{equation*}
    d_{-}(-e-d_{-})+d_{+}(e-d_{+})+d_{-}d_{+}.
  \end{equation*}
Hence the functor
\begin{equation*}
  r^{+}_{*}r^{-*}:Perf(Gr_{X}(\mc{E}^{\vee}[-1],d_{-}))\to Perf(Gr_{X}(\mc{E},d_{+}))
\end{equation*}
is represented by the Fourier-Mukai transform
\begin{equation*}
 \Phi_{d_{+},d_{-}}:=(r^{+}\times r^{-})_{*}(\mc{O}_{Inc_{X,d_{+},d_{-}}})\in QCoh(Gr_{X}(\mc{E}^{\vee}[-1],d_{-})\times Gr_{X}(\mc{E},d_{+})), 
\end{equation*}
 Moreover, let $\Psi_{d_{+},d_{-}}^{L}$ and $\Psi_{d_{+},d_{-}}$ be the left and adjoint functor of $\Phi_{d_{+},d_{-}}$. By the Grothendieck duality \cref{dual}, we have
\begin{equation*}
  \Psi_{d_{+},d_{-}}^{L}\cong (r^{-}\times r^{+})_{*} (K_{r^{-}}[vdim (r^{-})]), \quad  \Psi_{d_{+},d_{-}}\cong (r^{-}\times r^{+})_{*} (K_{r^{+}}[vdim (r^{+})])
\end{equation*}
as the Fourier-Mukai kernels.
\begin{remark}
  We will not calculate $K_{r^-}$ and $K_{r^+}$ explicitly in this paper as we do not need it. The reader can refer to \cite{jiang2022derived} for the explicit formula. However, we should point out that they are generated by tautological complexes over the incidence variety.
\end{remark}
\subsection{Semi-orthogonal decomposition and decomposition of the diagonal}
When $e:=rank(\mc{E})\geq 0$ is positive, by Toda \cite{MR4608555} and Jiang \cite{jiang2023derived}, for any $d\geq 0$ the category
\begin{equation*}
  Perf(Gr_{X}(\mc{E},d))
\end{equation*}
has a semi-orthogonal decomposition by copies of
\begin{equation*}
  Perf(Gr_{X}(\mc{E}^{\vee}[1],d')), \quad max\{d-e, 0\}\leq d'\leq d.
\end{equation*}
Moreover, Jiang \cite{jiang2023derived} proved that the functors are given as Fourier-Mukai transforms through the incidence varieties. Now we introduce a special case when $e=1$.
\begin{theorem}[\cite{jiang2023derived}]
  \label{thm1}
  If $e=1$, then for any $d>0$, the functors $\Phi_{d,d}$ and $\Phi_{d.d-1}$ are all fully faithful and we have the semi-orthogonal decomposition
  \begin{equation*}
    Perf(Gr_{X}(\mc{E},d))=<Im(\Phi_{d,d-1}),Im(\Phi_{d,d})>.
  \end{equation*}
\end{theorem}

\begin{corollary}
  Let $\Psi_{d,d-1}$ be the right adjoint functor of $\Phi_{d,d-1}$ and $\Psi_{d,d}^{L}$ be the left-adjoint fucntor of $\Phi_{d,d}$. Then there exists a canonical triangle of Fourier-Mukai transforms
\begin{equation*}
    \Phi_{d,d-1}\Psi_{d,d-1}\to \mc{O}_{\Delta}\to \Phi_{d,d}\Psi_{d,d}^{L}
\end{equation*}
where $\Delta$ is the diagonal embedding of $Gr_{X}(\mc{E},d)$. Moreover, we have
\begin{equation*}
    \Psi_{d,d-1}\Phi_{d,d-1}\cong \mc{O}_{\Delta}, \quad \Psi_{d,d}^{L}\Phi_{d,d}\cong \mc{O}_{\Delta}, \quad \Psi_{d,d-1}\Phi_{d,d}\cong 0, \quad \Psi_{d,d}^{L}\Phi_{d,d-1}\cong 0.
\end{equation*}
\end{corollary}
\begin{proof}trans
  We only need to show natural transformations of functors are induced by morphisms of Fourier-Mukai kernels. As all the settings can be naturally lifted to $\infty$-categories, it follows from Theorem 1.2(2) of \cite{MR2669705}.
\end{proof}

From now on, we denote 
\begin{equation*}
  \Psi_{d,d-1}:=\Psi_{d,d-1}^R, \quad \Psi_{d,d}:=\Psi_{d,d}^{L}.
\end{equation*}
\begin{corollary}
  \label{cor:23}
  Let $[\Psi_{d,d-1}]$, $[\Psi_{d,d}]$, $[\Psi_{d,d-1}]$, $[\Psi_{d,d}]$ be the corresponding correspondences at the level of Grothendieck group of perfect complexes. Then we have
  \begin{equation}
    \label{eq21}
    [\Psi_{d,d-1}][\Phi_{d,d-1}]=id=[\Psi_{d,d}][\Phi_{d,d}] \quad [\Psi_{d,d-1}][\Phi_{d,d}]=0=[\Psi_{d,d}][\Phi_{d,d-1}],
  \end{equation}
and 
\begin{equation}\label{eq22} 
  [\Phi_{d,d-1}][\Psi_{d,d-1}]+[\Phi_{d,d}][\Psi_{d,d}]=id.
\end{equation}
\end{corollary}

\section{Mukai pairing, K-theoretic and cohomological correspondences}
Let $\mb{H}$ be a cohomology theory other than the Grothendieck groups of coherent sheaves. In this section, we always assume that $X$ and $Y$ are smooth projective varieties. Given an element in the Grothendieck group of coherent sheaves $f\in K(X\times Y)$, one can define the K-theoretic correspondence
\begin{equation*}
  \Gamma_{X\to Y}^{\mu}:K(X)\to K(Y), \quad \Gamma_{X\to Y}^{\mu}(\cdot)=\pi_{Y,*}(\pi_{X}^{*}(\cdot)\otimes \mu).
\end{equation*}
Similarly, for any element $u\in \mb{H}^{*}(X\times Y,\mb{Q})$, one can define the cohomological correspondence
\begin{equation*}
  \gamma_{X\to Y}^{u}:\mb{H}^{*}(X,\mb{Q})\to \mb{H}^{*}(Y,\mb{Q}), \quad \gamma_{X\to Y}^{u}(\cdot)=\pi_{Y*}(\pi_{X}^{*}(\cdot).\mu).
\end{equation*}
Here $\mb{H}^{*}$ can be one of the Chow groups $CH^{*}$, the singular cohomology $H^{*}$, the Hodge cohomology $H^{*,*}$ or the Hochschild homology $HH_{*}$.

One may expect that two transforms can be naturally compared by the Chern character, i.e. we have a commutative diagram
\begin{equation}
  \label{eq26}
  \begin{tikzcd}
    K(X)\ar{r}{\Gamma_{X\to Y}^{\mu}}\ar{d}{ch} & K(Y)\ar{d}{ch} \\
    \mb{H}^{*}(X,\mb{Q})\ar{r}{\gamma_{X\to Y}^{ch(\mu)}} & \mb{H}^{*}(Y,\mb{Q})
  \end{tikzcd}
\end{equation}
where $ch$ is the Chern character. Unfortunately, \cref{eq26} in general is not true due to the Grothendieck-Riemann-Roch theorem. Caldararu \cite{1362544419271316864} found the correct way to modify the above diagram by generalizing the Mukai vectors, which we will recall in this section.

\subsection{K-theoretic and cohomological correspondences}
By the cohomological grading, we have a decomposition based on the odd or even degree
\begin{equation*}
  H^{*}(X,\mb{Q})\cong H^{even}(X,\mb{Q})\oplus H^{odd}(X,\mb{Q})
\end{equation*}
where $H^{even}(X,\mb{Q})$ is a commutative ring.  By Section 2.4 of \cite{1362544419271316864}, there exists a unique formal series expansion such that
\begin{equation*}
  \sqrt{1}=1,\quad \sqrt{uv}=\sqrt{u}\sqrt{v},\quad u=(\sqrt{u})^{2}
\end{equation*}
for every smooth projective variety $X$ and any $u,v\in H^{even}(X,\mb{Q})$ (resp. $CH^{*}(X,\mb{Q})$ and $HH_{0}(X)$) with constant term $1$. We define $\hat{A}_{X}\in \mb{H}^{*}(X,\mb{Q})$ as
\begin{equation*}
  \hat{A}_{X}:=td_{X}\cdot \sqrt{ch(K_{X})}
\end{equation*}
where $K_{X}$ is the canonical line bundle of $X$ and $td_{X}$ is the Todd class of $X$.

\begin{definition}[Definition 2.1 of \cite{1362544419271316864}]
  The directed Mukai vector of an element $\mu\in K(X\times Y)$ is defined by
  \begin{equation*}
    v(\mu,X\to Y):=ch(\mu)\cdot \sqrt{td_{X\times Y}}\cdot \sqrt[4]{ch(K_{Y})/ch(K_{X})}\in \mb{H}^{*}(X\times Y)
  \end{equation*}

  By taking the first space to be a point we obtain the definition of the Mukai vector of an object $\alpha\in K(X)$ as
  \begin{equation*}
    v(\alpha)=ch(\alpha)\sqrt{\hat{A}_{X}}\in \mb{H}^{*}(X).
  \end{equation*}
\end{definition}

\begin{definition}
  We define a map $(-)_{*}$ from the $K$-theoretic correspondences on $K(X\times Y)$ to cohomological correspondences $\mb{H}^{*}(X\times Y,\mb{Q})$ by
  \begin{equation*}
    \Gamma_{X\to Y *}^{\mu}:=\gamma_{X\to Y}^{v(\mu)}.
  \end{equation*}
\end{definition}

\begin{proposition}[Section 2.3 of \cite{1362544419271316864}]
  The map $(-)_{*}$ is linear with  the following properties:
  \begin{enumerate}
  \item Given correspondences $\mu:K(X)\to K(Y)$, $\nu:K(Y)\to K(Z)$, we have
    \begin{equation*}
      \nu_{*}\circ \mu_{*}\cong \nu_{*}\circ \mu_{*};
    \end{equation*}
  \item $id_{*}\cong id$;
  \item The following commutative diagram commutes:
    \begin{equation}
  \label{eq27}
  \begin{tikzcd}
    K(X)\ar{r}{\Gamma_{X\to Y}^{\mu}}\ar{d}{v} & K(Y)\ar{d}{v} \\
    \mb{H}^{*}(X,\mb{Q})\ar{r}{\gamma_{X\to Y}^{v(\mu)}} & \mb{H}^{*}(Y,\mb{Q})
  \end{tikzcd}
\end{equation}
  \end{enumerate}
\end{proposition}

\subsection{Cohomological correspondences of derived Grassmannians}

Now we come back to the setting of \cref{sec:2.2}. Let $X$ be a smooth proper variety with a tor-amplitude $[0,1]$ perfect complex $\mc{E}$. We assume that for all $d\geq 0$, the derived Grassmannian
\begin{equation*}
  Gr_{X}(\mc{E},d), \quad Gr_{X}(\mc{E}^{\vee}[1],d)
\end{equation*}
are (classical) smooth varieties. We apply the Caldararu's construction for $[\Phi_{d_{+},d_{-}}]$ and by \cref{eq21} and \cref{eq22}  and have
\begin{corollary} 
  \label{cor1111} The cohomological correspondences $[\Phi_{d,d-1}]_*,[\Psi_{d,{d-1}}]_*,[\Phi_{d,d}]_*,[\Psi_{d,d}]_*$ satisfy the relations \cref{eq:EF}, if we denote
  \begin{align*}
   E:=[\Phi_{d,d-1}]_*,\quad F:=[\Psi_{d,d-1}]_*, \quad K_{1}:=[\Phi_{d,d}]_*, \quad K_{-1}:=[\Phi_{d,d}]_*.
  \end{align*}
\end{corollary}

The correspondences in \cref{cor1111} are homogeneous with degree $0$ for the Hochschild homology, but not for the singular cohomology or Chow groups and we need to modify them. Recall that the incidence variety
\begin{equation*}
  Inc_{X}(\mc{E},d_{+},d_{-})
\end{equation*}
is quasi-smooth with virtual dimension $d_{0}=dim(X)+(1-d_{+})d_{+}-(d_{-}+1)d_{-}+d_{-}d_{+}$. By Ciocan-Fontanine and Kapranov \cite{ciocan-fontanine07:virtual}, the derived structure induces a perfect obstruction theory on the classical scheme of $Inc(X,d_{+},d_{-})$ and thus induces a virtual fundamental class in the sense of Li-Tian \cite{MR1467172} or Behrend-Fantechi \cite{MR1437495}:
$$Vir_{d_{+},d_{-}}:=[Inc_{X}(\mc{E},d_{+},d_{-})]^{vir}\in CH_{d_{0}}(Inc_{X}(\mc{E},d_{+},d_{-})).$$
Now we define cohomological correspondences $\psi_{d,d-1},\psi_{d,d},\phi_{d,d-1},\phi_{d,d}$ such that
\begin{enumerate}
\item $\phi_{d,d}$ is induced by $Vir_{d,d}$, and $\psi_{d,d}$ is induced by $(-1)^{d}Vir_{d,d}$ but in a different direction. 
\item both $\phi_{d,d-1}$ and $\psi_{d,d-1}$ are induced by $Vir_{d,d-1}$ but in two different directions. 
\end{enumerate} 

\begin{proposition}
  \label{prop:3155}
  The operators $\phi_{d,d-1}$ and $\psi_{d,d-1}$ preserved the cohomological degree and the operators $\phi_{d,d}$ and $\psi_{d,d}$ shifted the Chow group (resp. cohomological, Hodge) degree by $d$ (resp. $2d$, $(d,d)$) and $-d$ (resp. $-2d$, $(-d,-d)$). Moreover, the cohomological correspondences $\phi_{d,d-1},\psi_{d,s-1},\phi_{d,d},\psi_{d,d}$ also satisfy the relations \cref{eq:EF}  if we denote
  \begin{align*}
   E:=\phi_{d,d-1},\quad F:=\psi_{d,d-1}, \quad K_{1}:=\phi_{d,d}, \quad K_{-1}:=\phi_{d,d}.
  \end{align*}

\end{proposition}
\begin{proof}
  First, we notice that all the constant terms of roots of the Todd class and Chern classes of line bundles are $1$. By the virtual Grothendieck-Riemann-Roch theorem (i.e. Theorem 3.23 and Corollary 3.25 of \cite{khan2019virtual}), we also have
  \begin{equation*}
    ch(\mc{O}_{Inc_{X}(\mc{E},d_{+},d_{-})})=Vir_{d_{+},d_{-}}+C_{d_{+},d_{-}}
  \end{equation*}
  where $C_{d_{+},d_{-}}$ have higher degrees than $Vir_{d_{+},d_{-}}$. Thus we have $[\Phi_{d,d}]_*-\phi_{d,d}$ has a higher cohomological degree than $\phi_{d,d}$. Similar computations also holds for $[\Phi_{d,d-1}]_*-\phi_{d,d-1},[\Psi_{d,d}]_*-\psi_{d,d}$ and $[\Psi_{d,d-1}]_*-\psi_{d,d-1}$. Thus by checking the cohomological grading, the relations  \cref{eq:EF} for  $\phi_{d,d-1},\psi_{d,d-1},\phi_{d,d},\psi_{d,d}$ follows from \cref{cor1111}.
\end{proof} 

In a companion paper \cite{zhao2024}, we will prove that
\begin{proposition}
  The incidence variety
  \begin{equation*}
  Inc_{X}(\mc{E},d_{+},d_{-})
\end{equation*}
is classical and
$$Vir_{d_{+},d_{-}}:=[Inc_{X}(\mc{E},d_{+},d_{-})]\in CH_{d_{0}}(Inc_{X}(\mc{E},d_{+},d_{-})).$$
\end{proposition}

\section{Perverse Coherent Sheaves and their moduli spaces}
\label{sec4}
Let $S$ be a smooth projective surface. We denote by
\begin{equation*}
  p:\hat{S}\to S
\end{equation*}
the blow up of $S$ at a closed point $o\in S$. We denote $C:=p^{-1}(o)$ as the exceptional (rational) curve. We denote
\begin{equation*}
  \mc{O}_{C}(m):=\mc{O}(-mC)|_{C}
\end{equation*}
which is the unique degree $m$ line bundle on the rational curve $C$ for any $m\in \mb{Z}$. Given an ample line bundle $H$ on $S$, Nakajima-Yoshioka defined the perverse coherent sheaf as
\begin{definition}
  \label{def:41}
  A stable perverse coherent sheaf $E$ on $\hat{S}$ with respect to $H$ is a coherent sheaf such that
    \begin{equation*}
      Hom(E,\mc{O}_{C}(-1))=0,\quad Hom(\mc{O}_{C},E)=0
    \end{equation*}
    and $p_{*}E$ is Gieseker stable with respect to $H$.

    Given a coherent sheaf $E$ on $\hat{S}$ and an integer $m$, we define $E$ to be $m$-stable if $E\otimes \mc{O}(-mc)$ is a perverse stable sheaf.
\end{definition}

\begin{remark}
   When the rank of $E$ is $1$, $p_{*}E$ is stable with respect to $H$ if and only if $p_{*}E$ is torsion-free, and hence independent with the choice of $H$.
\end{remark}

By the work of  Nakajima-Yoshioka \cite{1361418519756759552,1050282810787470592,1050564285764182144}, the stable perverse coherent sheaf builds a bridge between the Hilbert scheme of points on $S$ and $\hat{S}$ (and more generally the moduli space of stable sheaves), which we recall in this section.

\subsection{Moduli space of perverse coherent sheaves}

We recall the decomposition of cohomology groups:
\begin{equation*}
  H^{*}(\hat{S}, \mb{Q})=p^{*}(H^{*}(\hat{S},\mb{Q}))\oplus \mb{Q}[C]
\end{equation*}
where $[C]\in H^{1,1}(\hat{S})\cap H^{2}(\hat{S},\mb{Q})$ is the fundamental class of $C$ satisfying $[C]^{2}=-1$.
\begin{definition}
  Given Chern characters $\hat{c}\in H^{*}(\hat{S},\mb{Q})$, we denote $M^{m}(\hat{c})$ as the moduli space of $m$-stable coherent sheaves on $\hat{S}$ with Chern character $\hat{c}$. Particularly, given integers $l,n$, we denote
  \begin{equation*}
    M^{m}(l,n):=M^{m}([\hat{S}]-l[C]+(-n-\frac{l^{2}}{2})[o]),
  \end{equation*}
  i.e. it consists of $m$-stable perverse coherent sheaves $E$ such that
  \begin{equation*}
    rank(E)=1, \quad c_{1}(E)=-l[C], \quad c_{2}(E)=n
  \end{equation*}
\end{definition}

By tensoring $\mc{O}(-mC)$, we have a canonical isomorphism
\begin{equation*}
  M^{m}(\hat{c})\cong M^{0}(\hat{c}e^{-m[C]}), \quad M^{m}(l,n)\cong M^{0}(l+m,n)
\end{equation*}

\begin{lemma}[Lemma 3.3 of \cite{koseki2021categorical} and Proposition 3.37 of \cite{1050282810787470592} ]
  For any integers $l,n$, the variety $M^{0}(l,n)$ is either empty or a smooth projective variety of $2n$ dimension. Particularly, fixing $n$, we have
  \begin{equation*}
    M^{0}(0,n)\cong S^{[n]},\quad M^{0}(l,n)\cong \hat{S}^{[n]} \text{ if } l\gg 0.
  \end{equation*}
\end{lemma}

\subsection{Moduli spaces as derived Grassmannians}
We first recall the cohomological properties of stable coherent sheaves
\begin{proposition}[Proposition 3.1 of \cite{kuhn2021blowup} and Lemma 3.4 of \cite{1050282810787470592}]
  \label{prop31}
  For a stable perverse coherent sheaf $E$ on $\hat{S}$, we have
  \begin{equation*}
    R^{i}p_{*}E(aC)=0, \quad a=0,1, \quad i\geq 1
  \end{equation*}
  and moreover
  \begin{equation*}
    Rp_{*}E(aC)=p_{*}E(aC), \quad a=0,1
  \end{equation*}
  are all stable coherent sheaves with respect to $H$.
\end{proposition}

For any stable perverse coherent sheaf $E$ such that
\begin{equation*}
    rank(E)=1, \quad c_{1}(E)=-l[C], \quad c_{2}(E)=n,
  \end{equation*}
  by \cref{prop31}
  \begin{equation*}
    p_{*}E, \text{ and } p_{*}E(C)
  \end{equation*}
are all stable and by the Grothendieck-Riemann-Roch theorem we have
  \begin{align*}
    rank(p_{*}E)=1, \quad c_{1}(p_{*}E)=0, \quad c_{2}(p_{*}E)=(n+\frac{l+l^{2}}{2}).\\
    rank(p_{*}E(C))=1, \quad c_{1}(p_{*}E(C))=0, \quad c_{2}(p_{*}E(C))=(n+\frac{-l+l^{2}}{2}).
  \end{align*}
  It induces morphisms of moduli spaces
\begin{align*}
  \zeta: M^{0}(l,n)\to S^{[n+\frac{l+l^{2}}{2}]}, \quad E\to p_{*}(E) \\
  \eta: M^{0}(l,n)\to S^{[n+\frac{-l+l^{2}}{2}]}, \quad E\to p_{*}(E(C)).
\end{align*}

The morphisms $\zeta$ and $\eta$ are actually projection morphisms of derived Grassmannians of the universal ideal sheaves over Hilbert schemes of points on $S$: let
\begin{equation*}
\mc{I}_{m}\in Coh(S^{[m]}\times S)  
\end{equation*}
be the universal ideal sheaf of $S^{[m]}$ and we denote
\begin{equation*}
  \mc{I}_{m,o}:=\mc{I}_{m}|_{S^{[m]}\times o},\quad (\text{short for }\mc{I}_{o}).
\end{equation*}
By a folklore lemma (like Proposition 2.14 of \cite{neguct2018hecke}), $\mc{I}_{o}$ is of tor amplitude $[0,1]$.

\begin{theorem}[Theorem 4.1 of \cite{1050282810787470592}]
  \label{thm35}
  We have a canonical isomorphism
  \begin{equation*}
    M^{0}(l,n)\cong Gr_{S^{[n+\frac{-l+l^{2}}{2}]}}(\mc{I}_{o},l), \quad M^{0}(l,n)\cong Gr_{S^{[n+\frac{l+l^{2}}{2}]}}(\mc{I}_{o}^{\vee}[1],l),
  \end{equation*}
  where the projection morphisms are $\eta$ and $\zeta$ exactly. 
\end{theorem}


\section{The admissible $\mc{E}$-representation on perverse coherent sheaves}

Given a smooth projective surface $S$ and a closed point $o\in S$ and a cohomology theory $\mb{H}$, we consider the grading which we call as the $l$-grading
\begin{equation*}
  \mb{H}^{S,o}_{l}:=\bigoplus_{n\in \mb{Z}}\mb{H}^{*}(M^{0}(l,n),\mb{Q}), \quad \mb{H}^{S,o}:=\bigoplus_{l\in \mb{Z}}\mb{H}^{S,o}_{l}.
\end{equation*}
When $l=0$, we have
\begin{equation*}
  \mb{H}_{0}^{S,o}\cong \bigoplus_{n\in \mb{Z}}\mb{H}^{*}(S^{[n]},\mb{Q}).
\end{equation*}

Other than the $l$-grading, we also have two grading: the $q$-degree by associating $\mb{H}^{*}(M^{0}(l,n),\mb{Q})$ with degree $n$, the $t$-degree (or bi-degree) which is the cohomology degree. We denote
\begin{equation*}
  1_{\mb{H}}:=
  \begin{cases}
    0, \text{ if } \mb{H} \text{ is the Hochschild homology group}, \\
    1, \text{ if } \mb{H} \text{ is the Chow group or Borel-Moore homology group}, \\
    (1,1)\text{ if } \mb{H} \text{ is the Hodge cohomology group.} \\
    \text{Undefined} \text{ if } \mb{H} \text{ is the Grothendieck group of coherent sheaves.}
  \end{cases}
\end{equation*}

In this section, we prove that
\begin{theorem}[\cref{thm:15}]
  \label{lastmain}
  There is a $\mc{E}$-representation on $\mb{H}^{S,o}$ such that
\begin{enumerate}
\item Under the $l$-grading, the representation is graded admissible;
\item the $q$-degree of $E,F$ are preserved, and the $q$-degree of $K_{\pm 1}$ are $\pm l$ on $\mb{H}^{S,o}_{l}$
\item the $t$-degree of $E,F$ are preserved, and the $t$-degree $K_{\pm 1}$ is $\pm l\cdot 1_{\mb{H}}$ on $\mb{H}^{S,o}_{l}$.
\item the stable limit
  \begin{equation*}
    \mb{H}_{\infty}^{S,o}\cong \bigoplus_{n\in \mb{Z}}\mb{H}^{*}(\hat{S}^{[n]},\mb{Q}).
  \end{equation*}
\item the operators are given by geometric correspondences.
\end{enumerate}
\end{theorem}

Once we have the above graded admissible representation, by the categorical Boson-Fermion correspondence theorem (i.e. \cref{fermion-boson}), we have
\begin{theorem}[\cref{thm2}]
  \label{thm62}
  We have an identification 
  \begin{equation*}
    \mb{H}^{S,o}\cong \mb{H}_{0}^{S,o}\otimes_{\mb{Q}} F.
  \end{equation*}
  and the action of the infinite-dimensional Clifford algebra on $F$ is given by geometric correspondences. The $t$-degree of a monomial in $F$:
  \begin{equation*}
    x_{i_{1}}\wedge \cdots x_{i_{n}}
  \end{equation*}
  is $(i_{1}+\cdots +i_{n}-\frac{n^{2}+n}{2})\cdot 1_{\mb{H}}$ and the $q$-degree is $i_{1}+\cdots +i_{n}-\frac{n^{2}+n}{2}$.
\end{theorem}
\begin{theorem}[\cref{thm:main}]
  \label{thm63}
  Let $\mathbb{H}^{*}$ be the functor of singular cohomology groups, Hochschild homology groups, or the Chow groups. We have an identification
  \begin{equation}
    \mb{H}_{\infty}^{S,o}\cong \mb{H}_{0}^{S,o}\otimes_{\mb{Q}}B.
  \end{equation}
  In particular, the Heisenberg algebra action on $B$ is induced by geometric correspondences. For a monomial $y_{\lambda}$ in the Bosonic Fock space, its $t$-degree is $|\lambda|\cdot 1_{\mb{H}}$ and the $q$-degree is $|\lambda|$.
\end{theorem}  

\subsection{The geometric correspondences}

We recall that over $S^{[n+\frac{-l+l^{2}}{2}]}$, by \cref{thm35} the derived Grassmannian of $\mc{I}_{o}$ induces
\begin{equation*}
  Gr(\mc{I}_{o}^{\vee}[1],l)\cong M^{0}(l,n-l), \quad Gr(\mc{I}_{o}^{\vee}[1],l-1)\cong M^{0}(l-1,n), \quad Gr(\mc{I}_{o},l)\cong M^{0}(l,n).
\end{equation*}
Thus the operators
\begin{equation*}
  \psi_{l,l-1}\quad \phi_{l,l-1}\quad \psi_{l,l}\quad \phi_{l,l}
\end{equation*}
acts on
\begin{align*}
  \psi_{l,l-1}:\bigoplus_{n\geq 0}\mb{H}^{*}(M^{0}(l,n),\mb{Q})\to\bigoplus_{n\geq 0} \mb{H}^{*}(M^{0}(l-1,n),\mb{Q}), \\
  \phi_{l-1,l}:\bigoplus_{n\geq 0}\mb{H}^{*}(M^{0}(l-1,n),\mb{Q})\to \bigoplus_{n\geq 0}\mb{H}^{*}(M^{0}(l,n),\mb{Q}), \\
  \psi_{l,l}:\bigoplus_{n\geq 0}\mb{H}^{*}(M^{0}(l,n),\mb{Q})\to \bigoplus_{n\geq 0}\mb{H}^{*}(M^{0}(l,n-l),\mb{Q}), \\
  \phi_{l,l-1}:\bigoplus_{n\geq 0}\mb{H}^{*}(M^{0}(l,n-l),\mb{Q})\to \bigoplus_{n\geq 0}\mb{H}^{*}(M^{0}(l,n),\mb{Q}).
\end{align*}
We define the operators
\begin{align}
  \label{eq31}
  E:=\bigoplus_{l\geq 0}\phi_{l-1,l}, \quad  F:=\bigoplus_{l\geq 0}\psi_{l-1,l}, \\
  K_{1}:=\bigoplus_{l\geq 0}\phi_{l,l}, \quad  K_{-1}:=\bigoplus_{l\geq 0}\psi_{l,l}. \nonumber
\end{align}
Then they all act on $\mb{H}_{l,n}^{S,o}$ (for the Grothendieck  Grothendieck or Hochschild homology, we replace it by $\Psi_{*,*}$ and $\Phi_{*,*}$ or $[\Psi_{*,*}]$ and $[\Phi_{*,*}]$ respectively).
\begin{proof}[Proof of \cref{lastmain}]
  By \cref{prop:3155}, operators in \cref{eq31} forms an $\mc{E}$-representation by geometric correspondences and the $q,t$-degree matches the \cref{lastmain}. The admissibility follows from the fact that
  \begin{equation*}
    M(l,n)=\emptyset
  \end{equation*}
  when $n<0$.
\end{proof}

\appendix
\section{Some algebra computations}
In this section, we give the proof of  \cref{prop12},  \cref{prop125},  \cref{224} and \cref{prop:228}, which involve explicit computations.
\subsection{The proof of \cref{prop12}}\label{A1} We recall the definition of the $P_{i}, Q_{i}$ for $i>0$ in the algebra $\mc{E}$:

\begin{definition}
  We define elements $P_{i},Q_{i}$ in $\mc{E}$ for any $i>0$ such that $P_{1}=EK_{-1}, Q_{1}=K_{1}F$ and
  \begin{align}
    \label{Cli1}   P_{i}=K_{1}P_{i-1}K_{-1}-EP_{i-1}F, \quad Q_{i}=K_{1}Q_{i-1}K_{-1}-EQ_{i-1}F ,\quad \text{ if } i>1.
  \end{align}
\end{definition} 
\begin{proof}[Proof of \cref{prop12}]

  We denote
  \begin{equation*}
    P(z):=\sum_{i=1}^{\infty}P_{i}z^{i}, \quad Q(z):=\sum_{i=1}^{\infty}Q_{i}z^{i}.
  \end{equation*}
  Then $P(z),Q(z)$ are decided by $P_{1},Q_{1}$ and the equation
  \begin{align}
    \label{for}
    P(z)=z(P_{1}+K_{1}P(z)K_{-1}-EP(z)F) \\
    Q(z)=z(Q_{1}+K_{1}Q(z)K_{-1}-EQ(z)F).
  \end{align}
  We need to prove that
  \begin{align*}
    P(z)P(w)+P(w)P(z)=0, \\ Q(z)Q(w)+Q(w)Q(z)=0, \\ P(z)Q(w)+Q(w)P(z)=\frac{zw}{1-zw}.
  \end{align*}
  
By \cref{for}, we have
\begin{align}
  \label{eqabc}  
  P(z)K_{1}=z(E+K_{1}P(z)) \quad K_{-1}P(z)=zP(z)K_{-1}, \\
    Q(z)K_{1}=zK_{1}Q(z), \quad K_{-1}Q(z)=z(F+K_{-1}Q(z)), \nonumber
  \end{align}
  and 
\begin{align}
  \label{eqqq}
  P(z)E=-zEP(z), \quad FP(z)=z(K_{-1}-P(z)F),  \\
  Q(z)E=z(K_{1}-EQ(z)), \quad FQ(z)=-zQ(z)F. \nonumber
\end{align}
Moreover, we have
\begin{align*}
 & P_{1}P(z)=EK_{-1}P(z)=zEP(z)K_{-1}=-P(z)P_{1}, \\
 & P_{1}Q(z)=EK_{-1}Q(z)=zE(F+Q(z)K_{-1})\\
 & =-Q(z)P_{1}+z(EF+K_{1}K_{1})=-Q(z)P_{1}+z.
\end{align*}
and similarly
\begin{align*}
  Q_{1}P(z)=-P(z)Q_{1}+z, \quad Q_{1}Q(z)=-Q(z)Q_{1}.
\end{align*}

Hence we have
\begin{align}
  \label{eqqqw}
  P(z)P(w)=zw(P_{1}+K_{1}P(z)K_{-1}-EP(z)F)(P_{1}+K_{1}P(w)K_{-1}-EP(w)F)
 \\
  =zP_{1}P(w)+wP(z)P_{1}-zwP_{1}^{2}+zw(K_{-1}P(z)P(w)K_{1}+EP(z)P(w)F).\nonumber
\end{align}
We notice that $P_{1}^{2}=EK_{-1}EK_{1}=0$. Thus let $\frak{P}=P(z)P(w)+P(w)P(z)$. We have
\begin{align*}
  \frak{P}=z(P_{1}P(w)+P(w)P(z))+w(P_{1}P(z)+P(z)P_{1})+zw(K_{-1}\frak{P}K_{1}+F\frak{P}E) \\
  =zw(K_{-1}\frak{P}K_{1}+F\frak{P}E)
\end{align*}
By the boundary condition, the only solution of
\begin{align*}
  \frak{P}=zw(K_{-1}\frak{P}K_{1}+F\frak{P}E)
\end{align*}
is $0$ and thus
\begin{align*}
  P(z)P(w)=P(w)P(z).
\end{align*}
Similarly, we have
\begin{align*}
  Q(z)Q(w)=Q(w)Q(z).
\end{align*}

Similar to \cref{eqqqw}, we have
\begin{align*}
  P(z)Q(w)=zP_{1}Q(w)+wP(z)Q_{1}-zwP_{1}Q_{1}+zw(K_{-1}P(z)Q(w)K_{1}+EP_{-1}(z)P(w)F), \\
  Q(w)P(z)=wQ_{1}P(z)+zQ(w)P_{1}-zwQ_{1}P_{1}+zw(K_{-1}Q(z)P(w)K_{1}+EQ(z)P(w)F).
\end{align*}
We notice that
\begin{equation*}
  P_{1}Q_{1}=EF, \quad Q_{1}P_{1}=K_{-1}K
\end{equation*}
and thus $P_{1}Q_{1}+Q_{1}P_{1}=1$. Hence let $\frak{Q}:=P(z)Q(w)+Q(w)P(z)$ and we have
\begin{align*}
  \frak{Q}=zw(1+K_{-1}\frak{Q}K_{1}+E\frak{Q}F)
\end{align*}
and hence
\begin{align*}
 (\frak{Q}-\frac{zw}{1-zw})=zw(K_{-1}(\frak{Q}-\frac{zw}{1-zw})K_{1}+E(\frak{Q}-\frac{zw}{1-zw})F)
\end{align*}
and hence
\begin{align*}
  P(z)Q(w)+Q(w)P(z)=\frac{zw}{1-zw}.
\end{align*}

\end{proof}

\subsection{The $\mc{E}$-representation on a graded admissible $\mc{C}$-module}
\label{125}
Let $W$ be a graded admissible $\mc{C}$-module. We formally define
\begin{equation*}
  P(z):=\sum_{i=1}^{\infty}P_{i}z^{i}, \quad Q(z):=\sum_{i=1}^{\infty}Q_{i}z^{i}
\end{equation*}
For any operator $T$ on $W$, we consider 
\begin{align*}
  \Gamma(T)(z):=\sum_{i\in \mb{Z}}\Gamma_{i}(T)z^{i}
\end{align*}
such that
\begin{align*}
  \Gamma_{i}(T)=\sum_{\substack{m-n=i\\ m,n\geq 1}}P_{m}TQ_{n}
\end{align*}
which is well-defined due to the admissible property. We also formally written
\begin{equation*}
  \Gamma(T)(z):=P(z)TQ(z^{-1}).
\end{equation*}
We denote
\begin{equation*}
  \Lambda:=\Gamma_{0}(id).
\end{equation*}
\begin{lemma}
  \label{lemA4}
  We have $\Lambda|_{W_{i}}=i\cdot id$, which is an invertible operator when $i>0$.
\end{lemma}
To prove \cref{lemA4}, we first prove that 
\begin{lemma}
  \label{coreq}\label{lem:comm}
  Let $\alpha$ and $\beta$ be two degree $0$ operators on $W$. Then $\alpha=\beta$ if 
  \begin{equation*}
    \beta|_{W_{0}}=\alpha|_{W_{0}}\quad \text{ and } \quad  [Q_{l},\beta]=[Q_{l},\alpha] \text{ for all } l>0.
  \end{equation*}
\end{lemma}
\begin{proof}
  We prove that $\alpha-\beta=0$ on all $W_i$ by the induction on $i$. Assuming that $\alpha-\beta|_{W_{i}}=0$, then on $W_{i+1}$,
  \begin{equation*}
    Q_{l}(\alpha-\beta)|_{W_{i+1}}=(\alpha-\beta)Q_l|_{W_i}=0 \text{ for all } l>0.
  \end{equation*}
  Hence $\alpha-\beta=0$ on $W_{i+1}$ by the admissible property of $W$.
\end{proof}
\begin{proof}[Proof of \cref{lemA4}]
  We notice that
    \begin{align*}
      Q_{l}\Lambda=-\sum_{i=1}^{\infty}P_{i}Q_{l}Q_{i}+Q_{l}=\Lambda Q_{l}+ Q_{l}
    \end{align*}
    and thus $\Lambda$ and $\oplus_{i\geq 0}i\cdot id$ all satisfy the equation
    \begin{equation*}
      [Q_{l},*]=Q_{l}
    \end{equation*}
    for all $l>0$. Hence they are the same by \cref{coreq}. 
  \end{proof}
By induction there are unique operators $K_{1}$, $K_{-1}$ on $W$ satisfy the following relations
\begin{enumerate}
\item $K_{1},K_{-1}$ are identity on $W_{0}$;
\item we have
  \begin{equation*}
    \Lambda K_{1}=\Gamma_{1}(K_{1}), \quad \Lambda K_{-1}=\Gamma_{-1}(K_{-1}).
  \end{equation*}
\end{enumerate}
We also define
\begin{equation*}
  E:=P_{1}K_{1}, \quad F:=K_{-1}Q_{1}.
\end{equation*}
\begin{proposition}
  \label{prop125}
  The operators $K_{1},K_{-1},E,F$ induce a graded admissible $\mc{E}$-module on $W$.
\end{proposition}
To prove \cref{prop125}, we first prove the following lemmas

\begin{lemma}
  \label{lem128}\label{lem129}
  The operators $K_{1},K_{-1},E,F$ satisfy the equation \cref{eqabc}, i.e. 
    \begin{align*}
    Q_{j}K_{1}=K_{1}Q_{j-1}, \quad K_{-1}P_j=P_{j-1}K_{-1}, \text{ if } j>0 \\
    P_{j}K_{1}=K_{1}P_{j-1}, \quad K_{-1}Q_{j}=Q_{j-1}K_{-1}, \text{ if } j>1,
  \end{align*}
  where we denote $P_{0}=Q_{0}=0$.
  \end{lemma} 
\begin{proof}
  We only prove $Q_{j}K_{1}=K_{1}Q_{j-1}$ and other computations follow from the same method.  We notice that 
  \begin{align*}
   (\Lambda+id)Q_{j}K_{1}=Q_{j}\Lambda K_{1}=\sum_{l=1}^{\infty}Q_{j}P_{l+1}K_{1}Q_{l}=-\sum_{l=1}^{\infty}P_{l+1}(Q_{j}K_{1})Q_{l}+K_{1}Q_{j-1} \\
    \Lambda K_{1}Q_{j-1}=\sum_{l=1}^{\infty}P_{l+1}K_{1}Q_{l}Q_{j-1}=-\sum_{l=1}^{\infty}P_{l+1}(K_{1}Q_{j-1})Q_{l}
  \end{align*}
  Thus $\alpha:=Q_{j}K_{1}-K_{1}Q_{j-1}$ satisfy the equation that
  \begin{equation}
    \label{eq:alpha}
    (\Lambda+id)\alpha=-\Gamma_{1}(\alpha).
  \end{equation}
  By induction, the solutions of \cref{eq:alpha} is decided by $\alpha|_{W_{0}}$, and hence $\alpha=0$ follows from $\alpha|_{W_{0}}=0$.
\end{proof}

\begin{corollary}
  For any $j>0$ (resp. $j>1$), we have
  \begin{equation*}
    Q_{j}K_{-1}K_{1}=K_{-1}K_{1}Q_{j}, \quad  (\text{resp. }Q_{j}K_{1}K_{-1}=K_{1}K_{-1}Q_{j}),
  \end{equation*}
  i.e.
  \begin{equation}
    \label{eq:comment}
    [Q(z),K_{-1}K_{1}]=0, \quad [\frac{d}{dz}Q(z),K_{1}K_{-1}]=0.
  \end{equation}
\end{corollary}
\begin{proof}
    It directly follows from \cref{lem128}.
  \end{proof}

  \begin{corollary}
      \label{lem130}  \label{lem132}
    On $W$ we have
  \begin{equation*}
    K_{1}K_{-1}=Q_{1}P_{1},\quad K_{-1}K_{1}=id.
  \end{equation*}
   
\end{corollary}
\begin{proof}
  The second formula follows from \cref{eq:comment} and \cref{lem:comm}. For the first formula, let $\beta:=K_{1}K_{-1}-Q_{1}P_{1}$. We have
  \begin{align*}
    \beta|_{W_{0}}=0, \quad Q_{1}\beta=0, \text{ and }\quad [Q_{l},\beta]=0 \text{ for all } l>1.
  \end{align*}
  Then by induction and the similar method with \cref{lem:comm}, we also have $\beta=0$.
\end{proof}

\begin{proof}[Proof of \cref{prop125}]
  The formula $K_{-1}K_{1}=id$ follows from \cref{lem132}. We have
  \begin{align*}
    FE=K_{-1}P_{1}Q_{1}K_{1}=K_{-1}K_{1}K_{-1}K_{1}=id
  \end{align*}
  by \cref{lem130}. Also by \cref{lem130},
  \begin{equation*}
    K_{1}K_{-1}+EF=Q_{1}P_{1}+P_{1}Q_{1}P_{1}Q_{1}=Q_{1}P_{1}+P_{1}Q_{1}-P_{1}^{2}Q_{1}^{2}=id.
  \end{equation*}
  Finally, we have
  \begin{equation*}
    FK_{1}=K_{-1}P_{1}K_{1}=0, \quad K_{-1}E=K_{-1}Q_{1}K_{1}=0
  \end{equation*}
  by \cref{lem128}.
\end{proof}

\subsection{The proof of \cref{224}} \label{sec:A2}
Let $H_{\pm k}$, $k\in \mb{Z}-\{0\}$ be the generators of the infinite-dimensional Heisenberg algebra.

Let $W_{\infty}$ be an admissible $\mc{H}$-representation, with the filtration
\begin{equation*}
  W_{i}:=\bigcap_{l> i}ker(H_{l}).
\end{equation*} 
and the operators for every $k>0$
\begin{equation*}
  \mc{L}_{k}:=-\sum_{i=1}^{\infty}\frac{1}{k^{i}i!}H_{k}^{i-1}H_{-k}^{i}.
\end{equation*}

\begin{lemma}
  We have
  \begin{equation}
    \label{A6}
    \mc{L}_{k}H_{k}=id, \quad H_{-k}(1-H_{k}\mc{L}_{k})=0.
  \end{equation}
\end{lemma}
\begin{proof}
  By induction, for any integer $i\geq 0$, we have
  \begin{equation}
    \label{A5}
    H_{-k}H_{k}^{i}=H_{k}^{i}H_{-k}-ik, \quad H_{-k}^{i}H_{k}=H_{k}H_{-k}^{i}-ik.
  \end{equation}
  By \cref{A5} we have
  \begin{align*}
    \mc{L}_{k}H_{k}=-\sum_{i\geq 1}\frac{1}{k^{i}i!}H_{k}^{i-1}H_{-k}^{i}H_{k}=\sum_{i\geq 1}(\frac{1}{k^{i-1}(i-1)!}H_{k}^{i-1}H_{-k}^{i-1}-\frac{1}{k^{i}i!}H_{k}^{i}H_{-k}^{i})=1.
  \end{align*}
  and
  \begin{align*}
    & H_{-k}(1-H_{k}\mc{L}_{k})=H_{-k}(Id+\sum_{i=1}^{\infty}\frac{1}{k^{i}i!}H_{k}^{i}H_{-k}^{i})\\
    &= H_{-k}+\sum_{i=1}^{\infty}(\frac{1}{k^{i}i!}H_{k}^{i}H_{-k}^{i+1}-\frac{1}{k^{i-1}(i-1)!}H_{k}^{i-1}H_{-k}^{i})=H_{-k}-H_{-k}=0.
  \end{align*} 
\end{proof}
\begin{proof}[Proof of \cref{224}]
  We only need to prove that it forms a graded $\mc{E}$-representation, and the admissibility follows from the admissibility of $\mc{W}_\infty$.

  First, we notice that for any $k>0$, the restriction of 
  \begin{equation*}
    \mc{L}_k:W_k\to W_k
  \end{equation*}
  at $W_{k-1}$ is zero. Thus we have
  \begin{equation*}
    K_{-1}E=0, \quad FE=id.
  \end{equation*}
  The formula $K_{1}K_{-1}+EF=0$ follows from the definition of $K_{\pm 1}$ and $E,F$.

  On the other hand, by \cref{A6} we have
  \begin{equation*}
    K_{-1}K_{1}=id, \quad FK_{1}=(1-K_{1}K_{-1})K_{1}=0.
  \end{equation*}
  Thus $W$ is a graded $\mc{E}$-representation.
\end{proof}

\subsection{The proof of \cref{prop:228} }
\label{sec:A3}
Let $W$ be an admissible graded $\mc{E}$-representation. As the limit of \cref{eq:ind}, on $W_{l}$ ($l>0$) we have
\begin{equation}
    \label{lem:119}
  \sum_{l=0}^{\infty}E_{l}F_{l}=id.
\end{equation}

\begin{lemma}
  \label{lem:118}
  For any $i>0$, the operators $H_{i,i}$ and $H_{i,-i}$ satisfy the relation
  \begin{equation*}
    H_{i,i}H_{i-i}-H_{i,-i}H_{i,i}=i\cdot id.
  \end{equation*}
\end{lemma}
\begin{proof}We have
  \begin{align*}
    & H_{i,i}H_{i,-i}=-iK_{1}\sum_{l=0}^{\infty}K_{1}^{l}K_{-1}^{l+1}= i-\sum_{l=0}^{\infty}(K_{1})^{l}K_{-1}^{l}\\ & =i-\sum_{l=0}^{\infty}(K_{1})^{l}K_{-1}^{l+1}K_{1}= i+H_{i,-i}H_{i,i}.
  \end{align*}
\end{proof}

\begin{lemma}
  \label{lem:120}
  For any $i>0$ and $|j|<i$, we have
  \begin{equation*}
    H_{i,j}H_{i,i}=H_{i,i}H_{i,j},\quad H_{i,j}H_{i,-i}=H_{i,-i}H_{i,j}
  \end{equation*}
\end{lemma}
\begin{proof}
  First we prove that $H_{i,j}H_{i,i}=H_{i,i}H_{i,j}$. It follows from
  \begin{align*}
    H_{i,j}H_{i,i}=\sum_{l=0}^{\infty}E_{l}H_{i-1,j}F_{l}K_{1}=\sum_{l=0}^{\infty}E_{l+1}H_{i-1,j}F_{l}=H_{i,i}H_{i,j}
  \end{align*}
  by \cref{lem22}.
  Then we prove that $H_{i,j}H_{i,-i}=H_{i,-i}H_{i,j}$. It follows from
  \begin{align*}
    & H_{i,j}H_{i,-i}=-i\sum_{l,m\geq 0}E_{l}H_{i-1,j}F_{l}K_{1}^{m}K_{-1}^{m+1}=-i\sum_{l\geq m\geq 0}E_{l}H_{i-1,j}F_{l+1}, \\
    & H_{i,-i}H_{i,j}=-i\sum_{l,m\geq 0}K_{1}^{m}K_{-1}^{m+1}E_{l}H_{i-1,j}F_{l}=-i\sum_{l>m\geq 0}E_{l-1}H_{i-1,j}F_l\\
    & =-i\sum_{l\geq m\geq 0}E_{l}H_{i-1,j}F_{l+1}=H_{i,j}H_{i,-i}
  \end{align*}
    by \cref{lem22}.
\end{proof}

\begin{proposition}
  \label{prop:120}
  On $W_{i}$, the operators $H_{i,j}$ where $j\neq 0$ and $-i\leq j\leq i$ satisfy the relation
  \begin{equation}
    \label{eq:com}
    H_{i,j}H_{i,k}-H_{i,k}H_{i,j}=j\delta_{j,-k}id.
  \end{equation}
\end{proposition}
\begin{proof}
  We assume that $|j|\leq |k|$. We prove it by induction on $i$. It is trivial when $i=0$. Assuming it is true for $i-1$, then if $k=\pm i$, \cref{eq:com} follows from \cref{lem:118} and \cref{lem:120}. Otherwise,   by \cref{lem22} we have
  \begin{align*}
    H_{i,j}H_{i,k}=\sum_{l,m\geq 0}E_{l}H_{i-1,j}F_{l}E_{m}H_{i-1,k}F_{m}= E_{l}H_{i-1,j}H_{i-1,k}F_{l}.
  \end{align*}
  Thus we have
  \begin{align*}
    [H_{i,j},H_{i,k}]=\sum_{l=0}^{\infty}E_{l}[H_{i-1,j},H_{i-1,k}]F_{l}
    =j\delta_{j,-k}\sum_{l=0}^{\infty}E_{l}F_{l}=j\delta_{j,-k}id
  \end{align*}
  by \cref{lem:119}. 
\end{proof}
\begin{proof}[Proof of \cref{prop:228} ]
 It is the limit of \cref{prop:120} when $i\to \infty$.
\end{proof}

\bibliography{resolution.bib}
\bibliographystyle{plain}

\end{document}